\documentclass{amsart}
\usepackage[english]{babel}
\usepackage[latin1]{inputenc}

\usepackage[a4paper,top=3cm,bottom=3cm,left=2.5cm,right=2.5cm,%
bindingoffset=5mm]{geometry}

\usepackage{newlfont}
\usepackage{amsmath}
\usepackage{amscd}
\usepackage{amssymb}
\usepackage{amsthm}
\usepackage{mathrsfs}
\usepackage{graphicx}
\usepackage{stmaryrd}
\usepackage{latexsym}
\usepackage{braket}
\usepackage{enumerate}
\usepackage{booktabs}
\usepackage{array}
\usepackage{paralist}
\usepackage{tikz-cd}

\usepackage{pxfonts}

\newtheorem{Theorem}{Theorem}[section]
\newtheorem*{Theorem*}{Theorem}
\newtheorem{Lemma}[Theorem]{Lemma}
\newtheorem{Cor}[Theorem]{Corollary}
\newtheorem{Prop}[Theorem]{Proposition}
\theoremstyle{definition}
\newtheorem{Def}[Theorem]{Definition}
\newtheorem*{Def*}{Definition}
\newtheorem{Rem}[Theorem]{Remark}
\newtheorem{Ex}[Theorem]{Example}

\newtheorem{Que}[Theorem]{Question}
\newtheorem*{Que*}{Question}

\newcommand{\kk}{k}            
\newcommand{\mmod}[1]{\mathrm{mod}(#1)}   
\newcommand{\mmodgr}[1]{\mathrm{mod}_{\mathbb{Z}}(#1)}   
\newcommand{\proj}[1]{\mathrm{Proj}(#1)}   
\newcommand{\cmod}[1]{\mathrm{MCM}(#1)}   
\newcommand{\cmodgr}[1]{\mathrm{MCM}_{\mathbb{Z}}(#1)}   
\newcommand{\Ref}[1]{\mathrm{Ref}(#1)}   
\newcommand{\Gl}[1]{\mathrm{GL}(#1)}   
\newcommand{\Sl}[1]{\mathrm{SL}(#1)}   

\DeclareMathOperator{\Syz}{Syz}
\DeclareMathOperator{\Sym}{Sym}

\DeclareMathOperator{\rank}{rank}
\DeclareMathOperator{\Hom}{Hom}
\DeclareMathOperator{\freerank}{frk}

\DeclareMathOperator{\chara}{char}

\usepackage{scalerel}
\usepackage{stackengine}

\newcommand\reallywidetilde[1]{\ThisStyle{%
  \setbox0=\hbox{$\SavedStyle#1$}%
  \stackengine{-.1\LMpt}{$\SavedStyle#1$}{%
    \stretchto{\scaleto{\SavedStyle\mkern.2mu\sim}{.5467\wd0}}{.7\ht0}%
  }{O}{c}{F}{T}{S}%
}}
\title{The symmetric signature}

\author{Holger Brenner and Alessio Caminata}

\address{{\small Holger Brenner, Institut f\"ur Mathematik, Universit\"at Osnabr\"uck, Albrechtstrasse 28a, 49076 Osnabr\"uck, Germany}}
\email{{\small holger.brenner@uni-osnabrueck.de}}

\address{{\small Alessio Caminata, Institut f\"ur Mathematik, Universit\"at Osnabr\"uck, Albrechtstrasse 28a, 49076 Osnabr\"uck, Germany}}
\curraddr{ Institut de Math\'{e}matiques, Universit\'{e} de Neuch\^{a}tel\\Rue Emile-Argand 11, CH-2000 Neuch\^{a}tel, Switzerland}
\email{{\small alessio.caminata@unine.ch}}

\begin{document}


\begin{abstract}
We define two related invariants for a $d$-dimensional local ring $(R,\mathfrak{m},\kk)$ called syzygy and differential symmetric signature by looking at the maximal free splitting of reflexive symmetric powers of two modules: the top dimensional syzygy module $\Syz_R^{d}(\kk)$ of the residue field and the module of K\"ahler differentials $\Omega_{R/\kk}$ of $R$ over $\kk$.
We compute these invariants for two-dimensional ADE singularities obtaining $1/|G|$, where $|G|$ is the order of the acting group, and for cones over elliptic curves obtaining $0$ for the differential symmetric signature.
These values coincide with the F-signature of such rings in positive characteristic.
\end{abstract}

\maketitle

\section*{Introduction}
\par Let $(R,\mathfrak{m},\kk)$ be a Noetherian local ring of prime characteristic $p$ and dimension $d$ which is F-finite, and with algebraically closed residue field $\kk$.
For every $e\in\mathbb{N}$, let $q=p^e$, and let $^e\!R$ be the $R$-module which is equal to $R$ as abelian group and has scalar multiplication twisted via the $e$-th Frobenius homomorphism, that is $r\circ s:=r^{q}s$ for $r\in R$ and $s\in\ ^e\!R$.
We decompose it as  
\begin{equation*}
^e\!R=R^{a_q}\oplus M_q,
\end{equation*}
where the module $M_q$ has no free direct summands.
The number $a_q$ is also called the \emph{free rank} of $^e\!R$ and denoted by $\freerank_R(^e\!R)$, while if $R$ is a domain, one has that $q^{d}=\rank_R(^e\!R)$, the usual rank of an $R$-module.
Huneke and Leuschke \cite{HL02} defined the \emph{F-signature} of $R$ as the limit 
\begin{equation*}
 s(R):=\lim_{e\rightarrow+\infty}\frac{a_q}{q^{d}}.
\end{equation*}

 \par They proved that the limit defining the F-signature exists assuming that $R$ is Gorenstein and in other cases.
 Watanabe and Yoshida \cite{WY04} introduced the \emph{minimal relative Hilbert-Kunz multiplicity}. 
 Then, Yao \cite{Yao06} proved it actually coincides with the F-signature.
Finally,  Tucker \cite{Tuc12} proved that the F-signature exists for every reduced F-finite Noetherian local ring.

\par The F-signature is a real number between $0$ and $1$ and provides delicate information about the singularities of $R$. 
Two principal results in this direction are the fact that $s(R)=1$ if and only if the ring is regular \cite{WY00}, and that $s(R) > 0$ if and only if $R$ is strongly F-regular \cite{AL03}.

\par The main motivation for this paper is given by the following question. 

\begin{Que*}
Does there exist an invariant analogous to the F-signature which is independent of the characteristic and can be used to study singularities also in characteristic zero?
\end{Que*}

\par A possible natural attempt to define a characteristic zero version of the F-signature is the following.
If $R$ is a reduced $\mathbb{Z}$-algebra such that $\mathrm{Spec} \,R\rightarrow\mathrm{Spec} \,\mathbb{Z}$ is dominant, then for every prime number $p$ we consider its reduction mod $p$, $R_p := R \otimes_{\mathbb{Z}} (\mathbb{Z}/p\mathbb{Z})$, and compute the F-signature $s(R_p)$. 
One may ask whether the limit 
\begin{equation*}
\lim_{p\rightarrow+\infty} s(R_p)
\end{equation*}
exists, and use this limit to define a characteristic $0$ version of the F-signature.

\par Unfortunately, it is difficult to find an appropriate meaning and compute the previous limit, or even determine whether it exists at all. 
If $p_1$ and $p_2$ are distinct prime numbers, the two Frobenius homomorphisms are difficult to compare.
If $R$ is a domain then the rank of $^e\!R_{p_1}$ and the rank of $^e\!R_{p_2}$ are different, the first one being $p_1^e$.

\par In this paper we define a new invariant called \emph{symmetric signature} which works in any characteristic.
As inspiration for our construction, we considered  two important properties of the modules $^e\!R$ which we would like to keep.
The first fact is a  result of Kunz \cite{Kun69} which states that $^e\!R$ is a free module (for all or for some $e\in\mathbb{N}$) if and only if $R$ is a regular ring.
Next, if the ring $R$ is Cohen-Macaulay then $^e\!R$ is a maximal Cohen-Macaulay (MCM for short) module.

\par Let $(R,\mathfrak{m},\kk)$ be a Noetherian local $\kk$-algebra of dimension $d$.
We look for an $R$-module with the property that it is free if and only if $R$ is regular.
We have at least two modules with this property: the \emph{top-dimensional syzygy module of the residue field} $\Syz^{d}_R(\kk)$, and the  \emph{module of (K\"ahler) differentials}  $\Omega_{R/\kk}$ of $R$ over $\kk$ or its dual.

\par If $R$ is Cohen-Macaulay, then the module $\Syz^{d}_R(\kk)$ is MCM by the depth lemma, on the other hand $\Omega_{R/\kk}$ is not even reflexive in general.
For this reason we prefer to work with its reflexive hull, $\Omega_{R/\kk}^{**}$, which is called \emph{module of Zariski differentials of $R$ over $\kk$}.
The sign $(-)^*$ stands for the functor $\Hom_R(-,R)$.

\par We want to study an asymptotic behaviour, so we construct two classes of modules (one for $\Syz^{d}_R(\kk)$, and one for $\Omega_{R/\kk}^{**}$) indexed by a natural number $q$.
To do this we introduce the functor of $q$-th reflexive symmetric powers $\left(\Sym^q_R(-)\right)^{**}$, that is ordinary $q$-th symmetric powers of $R$-modules composed with the reflexive hull.

\par The choice of studying symmetric powers is motivated by an unpublished work of Brenner and Fischbacher-Weitz \cite{BF09}.
Symmetric powers of cotangent bundles on projective varieties have been studied also by other authors such as Bogomolov and De Oliveira\cite{BDO08}, and Sakai \cite{Sak78}.
Instead, the reflexive hull is motivated by the fact that we want to stay inside the category of reflexive modules, which coincides with the category of MCM $R$-modules if $R$ is a normal two-dimensional domain.

\par More precisely, we procede as follows.
 We consider $\Syz^{d}_R(\kk)$ (for $\Omega_{R/\kk}^{**}$ the situation is analogous) and we apply to it the $q$-th reflexive symmetric powers functor for a natural number $q$.
We obtain a reflexive $R$-module 
\begin{equation*}
\mathcal{S}^q:=\left(\Sym_R^q\big(\Syz^{d}_R(\kk)\big)\right)^{**} 
\end{equation*}
\par We decompose the module $\mathcal{S}^q$ as
\begin{equation*}
 \mathcal{S}^q= R^{a_q} \oplus M_q,
\end{equation*}
with the module $M_q$ containing no free direct summands, so that $a_q=\freerank_R\mathcal{S}^q$ the free rank of $\mathcal{S}^q$. 
For ease of notation we fix also $b_q:=\rank_R\mathcal{S}^q$, and we introduce the following (Definition \ref{defsymmetricsignature}).
\begin{Def*} The real number
\begin{equation*}
s_{\sigma}(R) := \lim_{N\rightarrow+\infty} \frac{\sum_{q=0}^Na_q}{\sum_{q=0}^Nb_q}
\end{equation*}
is called  \emph{(syzygy) symmetric signature} of $R$, provided the limit exists.
\end{Def*}
\par Replacing the module $\Syz^{d}_R(\kk)$ with $\Omega_{R/\kk}^{**}$ in the previous construction we define also the \emph{differential symmetric signature} of $R$, which we denote by $s_{d\sigma}(R)$.

\par In order to understand the behaviour of these new invariants in comparison with the F-signature we concentrate on two important classes of examples: two-dimensional Kleinian singularities, and  coordinate rings of plane elliptic curves.

\par Two-dimensional Kleinian or ADE singularities are the rings of invariants $R=S^G$ of a power series ring in two variables $S=\kk\llbracket u,v\rrbracket$ under the linear action of certain finite subgroups $G$ of $\Sl{2,\kk}$, called Klein groups.
The field $\kk$ is algebraically closed and such that its characteristic does not divide the order of the acting group $G$.
Watanabe and Yoshida \cite{WY04} proved that if the characteristic of $\kk$ is a prime number then the F-signature of $R$ is $s(R)=\frac{1}{|G|}$.
We prove that the same holds for the symmetric and the differential symmetric signature in any characteristic (Corollary \ref{corollaryKleinian}), that is
\begin{equation*}
 s_{\sigma}(R)=s_{d\sigma}(R)=\frac{1}{|G|}.
\end{equation*}

\par Our main tool to prove this result is the so-called \emph{Auslander correspondence} (Theorem \ref{Auslandertheorem}).
This states that for a Kleinian singularitiy $R=S^G$ there is a functorial one to one correspondence between irreducible $\kk$-representations of $G$ and indecomposable maximal Cohen-Macaulay $R$-modules.
The module corresponding to a $\kk$-representation $V$ is given by $\mathcal{A}(V):=(S\otimes_{\kk}V)^G$.

\par We prove  that this construction commutes with $q$-th reflexive symmetric powers (Theorem \ref{symcommuteswithauslandertheorem}), that is  we have
 \begin{equation*}
  \mathcal{A}\left(\Sym_{\kk}^q(V)\right)\cong\left(\Sym_R^q(\mathcal{A}(V))\right)^{**}.
 \end{equation*}
Here, $\Sym_R^q(-)$ denotes the $q$-th symmetric power of an $R$-module, and $\Sym_{\kk}^q(V)$ is the $q$-th symmetric power of the representation $V$.

\par This result is the key to translate the problem of computing the symmetric signature into a problem in representation theory of finite groups.
In fact, it turns out that the free rank of some MCM $R$-module $M=\mathcal{A}(V)$ is equal to the multiplicity of the trivial representation into $V$, and the latter can be more easily computed using tools from representation theory, such as characters.

\par The symmetric signature of two-dimensional cyclic quotient singularities $\kk\llbracket u,v\rrbracket^G$, with $G$ cyclic, is also $1/|G|$ and is computed by Katth\"an and the second author in the paper \cite{CK16}.

\par Let $R$ be the coordinate ring of a plane elliptic curve over an algebraically closed field $\kk$.
Then $R$ is not strongly F-regular, therefore its F-signature is $s(R)=0$.
We prove that the differential symmetric signature of $R$ is $s_{d\sigma}(R)=0$ (Theorem \ref{Theocotangentelliptic}), provided that $\chara\kk$ is not $2$ or $3$.
For the syzygy symmetric signature we give an upper bound, that is $s_{\sigma}(R)\leq\frac{1}{2}$ (Corollary \ref{corsleq12elliptic}), provided that the limit exists.

\par The methods we use in this situation are geometric. 
We use the correspondence between graded MCM $R$-modules and vector bundles over the smooth projective curve $Y=\mathrm{Proj}\,R$, and we translate the problem of computing $s_{\sigma}(R)$ and $s_{d\sigma}(R)$  into an analogous problem in the category $\mathrm{VB}(Y)$ of vector bundles over $Y$.
The main advantage of this approach is that over an elliptic curve $Y$ the indecomposable vector bundles have been classified and described by Atiyah \cite{Ati57}.

\par The structure of this paper is the following.
\par In Section \ref{sectionFandsymmetric} we give a short review of the F-signature and its properties.
Then, we define the symmetric and differential symmetric signatures, together with some other variants: a generalized version for modules (Definition \ref{defgeneralizedsymmetric}), and one for graded rings (Definition \ref{defgradedsymmetric}).
Some easy properties and consequences of the symmetric signature are also stated in this section.

\par In Section \ref{sectionAuslandersymmetric} we review the theory of the Auslander correspondence and we prove that the Auslander functor commutes with reflexive symmetric powers (Theorem \ref{symcommuteswithauslandertheorem}).

\par Section \ref{sectionKleinian} is dedicated to the computation of the symmetric signatures of two-dimensional Kleinian singularities, and Section \ref{sectionelliptic} to the computation of the symmetric signature for coordinate rings of plane elliptic curves.

\section{F-signature and Symmetric signature}\label{sectionFandsymmetric}
\subsection{F-signature}\label{subsectionFsignature}
\par We recall some general facts about rings of prime characteristic, and the definition and basic properties of the F-signature. 
Let $R$ be a commutative ring containing a field of prime characteristic $p$. 
With the letter $e$ we always denote a natural number, and with $q=p^e$ a power of the characteristic. 
The \emph{Frobenius homomorphism} is the ring homomorphism $F:R\rightarrow R$, $F(r)=r^p$, we often consider also its iterates $F^e:R\rightarrow R$, $F^e(r)=r^{q}$. 
For any finitely generated $R$-module $M$, we denote by $^e\! M$ the $R$-module $M$, whose multiplicative structure is pulled back via $F^e$. 
The scalar multiplication on $^e\!M$ is given by $r\cdot m:=r^qm$, for any $r\in R$, $m\in \ \! ^e\!M$. 

\par We say that $R$ is \emph{F-finite} if $^1\!R$ is a finitely generated $R$-module, and we say that $R$ is \emph{F-split} if the Frobenius map $F:R\rightarrow\ ^1\!R$ splits. 
If $(R,\mathfrak{m},\kk)$ is a complete Noetherian local ring, then $R$ is F-finite if and only if $[\kk^{1/p}:\kk]$ is finite. 
We will always assume that our rings are F-finite, and that $\kk$ is perfect, that is $[\kk^{1/p}:\kk]=1$. 
The latter is not a strong restriction, but it is done simply to avoid that numbers like $[\kk^{1/p}:\kk]$ appear in the definitions and in the results. 

\begin{Rem}\label{eRisMCM}
 If $R$ is F-finite and Cohen-Macaulay, then $ ^e\!R$ is a maximal Cohen-Macaulay module for every $e\in\mathbb{N}$. 
 Actually, if $x_1,\dots,x_n$ is a maximal $R$-regular sequence, then it is also an $ ^e\!R$-regular sequence for every $e\in\mathbb{N}$.
\end{Rem}

\par Let $(R,\mathfrak{m},\kk)$ be a $d$-dimensional Noetherian reduced local ring of prime characteristic $p$, which is F-finite and such that $\kk$ is perfect. 
We recall the definition of \emph{free rank} of an $R$-module $M$
\begin{equation*}
\freerank_R(M):=\max\{n: \ \exists \text{ a split surjection } \varphi:M\twoheadrightarrow R^n\}.
\end{equation*}
Since the cancellation property holds for finitely generated modules over a local ring (cf. \cite[Corollary 1.16]{LW12}), $\freerank_R(M)$ is the unique integer such that we have a decomposition
\begin{equation*}
M\cong R^{\freerank_R(M)}\oplus N,
\end{equation*}
and the module $N$ contains no free $R$-direct summands.

\begin{Def}[Huneke-Leuschke, \cite{HL02}]
 The \emph{F-signature} of $R$, denoted by $s(R)$ is the limit
 \begin{equation*}
  s(R):=\lim_{e\rightarrow+\infty}\frac{\freerank_R(^e\!R)}{q^d}.
 \end{equation*}
 \end{Def}
 Tucker  \cite{Tuc12} proved  that the F-signature exists for every reduced F-finite Noetherian local ring.
\begin{Rem}
 If $R$ is a domain of dimension $d$, then $\rank_R(^e\!R)=q^d$. So we can write the limit defining the F-signature as
 \begin{equation*}
  s(R)=\lim_{e\rightarrow+\infty}\frac{\freerank_R (^e\!R)}{\rank_R(^e\!R)}.
 \end{equation*}
\end{Rem}

\par The F-signature is always a real number in the interval $[0,1]$. 
Using a result of Watanabe and Yoshida \cite{WY00}, Huneke and Leuschke \cite{HL02} proved that if $R$ is Cohen-Macaulay then the extreme value $1$ is obtained if and only if the ring is regular. Then, Yao \cite{Yao06} removed the Cohen-Macaulay assumption. 

\begin{Theorem}\label{WatYoshTheorem1}
 Let $(R,\mathfrak{m},\kk)$ be a reduced F-finite local ring of prime characteristic such that $\kk$ is infinite and perfect. 
 Then $s(R)=1$ if and only if $R$ is regular.
\end{Theorem}

\par Also the value $s(R)=0$ has a special meaning in terms of the singularities of the ring, it is equivalent to the ring being not strongly F-regular, an important notion in tight closure theory.

\begin{Theorem}[Aberbach-Leuschke, \cite{AL03}]\label{AbeLeusTheorem}
 Let $(R,\mathfrak{m},\kk)$ be a reduced excellent F-finite local ring of prime characteristic such that $\kk$ is perfect. 
 Then $s(R)>0$ if and only if $R$ is strongly F-regular.
\end{Theorem}

\begin{Rem}
 If the ring $R$ has dimension $0$ or $1$, then it is normal if and only if it is regular if and only if it is strongly F-regular.
 So by Theorem \ref{WatYoshTheorem1} and Theorem \ref{AbeLeusTheorem} we have only two possibilities for the F-signature of $R$.
 We have $s(R)=1$ if $R$ is regular, and $s(R)=0$ otherwise.
\end{Rem}

Theorem \ref{WatYoshTheorem1} and Theorem \ref{AbeLeusTheorem} justify the statement that \emph{F-signature measures singularities}. 
Roughly speaking, the closer the F-signature to $1$ is, the nicer the singularity.
An important example is the F-signature of quotient singularities computed by Watanabe and Yoshida.

\begin{Theorem}[Watanabe-Yoshida, \cite{WY04}]\label{WatYoshTheorem2}
Let $R=\kk\llbracket x_1,\dots x_n\rrbracket^G$ be a quotient singularity over a field $\kk$ of prime characteristic $p$. 
Assume that the acting group $G\subseteq\Gl{n,\kk}$  is a small finite group such that $(p,|G|)=1$. 
Then the F-signature of $R$ is
\begin{equation*}
 s(R)=\frac{1}{|G|}.
\end{equation*}
\end{Theorem}

\par We conclude with the notion of generalized F-signature, introduced by Hashimoto and Nakajima \cite{HN15}. 
Let $\mathcal{C}$ be a full subcategory of $\mmod{R}$, the category of finitely generated $R$-modules, such that $\mathcal{C}$ has the Krull-Remak-Schmidt property (KRS property for short) and $^e\!R\in\mathcal{C}$ for every $e\in\mathbb{N}$. 
For example, if $R$ is complete we may choose $\mathcal{C}=\mmod{R}$.
Let $M$ be an indecomposable object in $\mathcal{C}$, for each $e\in\mathbb{N}$ and $q=p^e$ we can consider the multiplicity $a_{q}$ of $M$ inside $^e\!R$. 
In other words, we can write $^e\!R=M^{a_q}\oplus N_q$, with the module $N_q$ containing no copies of $M$ as direct summands.
\begin{Def}[Hashimoto-Nakajima, \cite{HN15}]
 The \emph{generalized F-signature of $R$ with respect to $M$} is
 \begin{equation*}
 s(R,M):=\lim_{e\rightarrow+\infty}\frac{a_q}{q^{d}},
 \end{equation*}
 provided the limit exists.
\end{Def}

\begin{Rem}
 If the ring $R$ has finite F-representation type, then the generalized F-signature exists. 
 This is a consequence of \cite[Proposition 3.3.1]{SVB96} and of \cite[Theorem 3.11]{Yao05}. 
 Moreover, Seibert  \cite{Sei97} proved  that in this situation the F-signature and the Hilbert-Kunz multiplicity of $R$ are rational numbers.
\end{Rem}

\begin{Theorem}[Hashimoto-Nakajima, \cite{HN15}]\label{YusukeTheorem}
 Let $R=\kk\llbracket x_1,\dots x_n\rrbracket^G$ be a quotient singularity over an algebraically closed field $\kk$ of prime characteristic $p$. 
 Assume that the acting group $G\subseteq\Gl{n,\kk}$  is a small finite group such that that $(p,|G|)=1$. 
 Let $V_t$ be an irreducible $\kk$-representation of $G$ and let $M_t=\mathcal{A}(V_t)=(S\otimes_{\kk}V_t)^G$ be the corresponding $R$-module via Auslander functor.  
 Then the generalized F-signature of $R$ with respect to $M_t$ is
\begin{equation*}
 s(R,M_t)=\frac{\rank_R M_t}{|G|}.
\end{equation*}
\end{Theorem}

The definition of the Auslander functor $\mathcal{A}$ and the relation between $\kk$-representations of $G$ and reflexive modules over the invariant ring $R$ will be reviewed in Section \ref{sectionAuslandersymmetric}.

\subsection{Symmetric signature}\label{subsectionsymmetricsignature}
\par For this section, let $(R,\mathfrak{m},\kk)$ be a local Noetherian $\kk$-domain of dimension $d$ (of any characteristic) and let $q$ be a natural number. 
We consider the top-dimensional syzygy module $\Syz^d_R(\kk)$ of the residue field, coming from a minimal free resolution of $\kk$, and the module of (K\"ahler) differentials of $R$ over $\kk$, denoted by $\Omega_{R/\kk}$.
\par It is an easy consequence of the Auslander-Buchsbaum-Serre theorem, and of Auslander-Buchsbaum formula that  $\Syz^d_R(\kk)$ is $R$-free if and only if $R$ is regular.
Similarly, under some mild conditions we have that $\Omega_{R/\kk}$ is $R$-free if and only if $R$ is regular (see e.g. \cite[Chapter II, Theorem 8.8]{Har77}).
 Moreover from the depth lemma follows that if $R$ is Cohen-Macaulay then $\Syz^d_R(\kk)$ is maximal Cohen-Macaulay, and in particular reflexive.
On the other hand $\Omega_{R/\kk}$ is not even reflexive in general (see e.g. \cite{Her78a}), so it is convenient to work with  its reflexive hull $\Omega_{R/\kk}^{**}$, which is called \emph{module of Zariski (or regular) differentials of $R$ over $\kk$}. 
\par Let $q$ be a natural number, we define the following two classes of $R$-modules.
\begin{equation*}
\begin{split}
\mathcal{S}^q:&=\left(\Sym_R^q\big(\Syz^d_R(\kk)\big)\right)^{**}, \\
  \mathcal{C}^q:&=\left(\Sym^q_R\big(\Omega_{R/\kk}^{**}\big)\right)^{**}=\left(\Sym^q_R\big(\Omega_{R/\kk}\big)\right)^{**}.
 \end{split}
 \end{equation*}

 \begin{Def}\label{defsymmetricsignature}
  The real numbers 
 \begin{equation*} 
  s_{\sigma}(R):=\lim_{N\rightarrow+\infty}\frac{\sum_{q=0}^N\freerank_R\mathcal{S}^q}{\sum_{q=0}^N\rank_R\mathcal{S}^q},
 \end{equation*}
 and 
 \begin{equation*}
  s_{d\sigma}(R):=\lim_{N\rightarrow+\infty}\frac{\sum_{q=0}^N\freerank_R\mathcal{C}^q}{\sum_{q=0}^N\rank_R\mathcal{C}^q}
 \end{equation*}
are called \emph{(syzygy) symmetric signature} of $R$, and \emph{differential symmetric signature} of $R$ respectively, provided the limits exist.
 \end{Def}

 \par In the rest of the paper we will study both the symmetric signature and the differential symmetric signature.
 For the clarity of exposition we will focus sometimes on one definition and we will say what applies also to the other.
 
 \par We don't know whether the limit defining the symmetric signature always exists.
To obtain an invariant which is for sure well defined and exists, one may replace the limits in Definition \ref{defsymmetricsignature} with limits inferior.
\par Moreover one may ask why should we consider the limits of Definition \ref{defsymmetricsignature}, insted of the simpler limit 
\begin{equation}\label{symmetricsimplerlimit}
\lim_{q\rightarrow+\infty} \frac{\freerank_R\mathcal{S}^q}{\rank_R\mathcal{S}^q}.
\end{equation}
The main reason is that the limit \eqref{symmetricsimplerlimit} does not exist even in simple cases, as we see in Example \ref{exampleAnlimitnotexists}.
However, when the simpler limit \eqref{symmetricsimplerlimit} exists, then also the symmetric signature exists and they coincide.

\begin{Lemma}
 Let $(a_n)_{n\in\mathbb{N}}$, $(b_n)_{n\in\mathbb{N}}$ be sequences of real numbers such that $b_n>0$ for all $n\in\mathbb{N}$. 
 Assume that the infinite series $\sum_{n\in\mathbb{N}}b_n$ diverges. If the limit $\lim_{n\rightarrow+\infty}\frac{a_n}{b_n}$ exists, then also the limit
 \begin{equation*}
  \lim_{n\rightarrow+\infty}\frac{\sum_{k=0}^na_k}{\sum_{k=0}^nb_k}
 \end{equation*}
exists and the two limits coincide.
\end{Lemma}

\begin{Rem}
 The modules $\mathcal{S}_q$ and $\mathcal{C}_q$ are reflexive for every $q\geq0$. 
 If $R$ is Cohen-Macaulay of dimension $\leq2$, then $\mathcal{S}_q$ and $\mathcal{C}_q$  are also MCM.
\end{Rem}

\begin{Ex}\label{exampleregularsymmetricis1}
 If $R$ is a regular ring of dimension $d$, then  $\Syz^d_R(\kk)$ and $\Omega_{R/\kk}^{**}$ are free modules. 
 Then $\mathcal{S}^q$ and $\mathcal{C}^q$ are also free, therefore $\freerank_R\mathcal{S}^q=\rank_R\mathcal{S}^q$ and $\freerank_R\mathcal{C}^q=\rank_R\mathcal{C}^q$for every $q$. 
 It follows that the symmetric signatures are $s_{\sigma}(R)=1$ and $s_{d\sigma}(R)=1$.
 In particular, this happens if $R$ is of dimension $0$, since it is forced to be a field, being a domain.
\end{Ex}

\begin{Rem}
If $(R,\mathfrak{m},\kk)$ has dimension $1$, then the first syzygy of the residue field $\Syz_R^1(\kk)$ is just the maximal ideal $\mathfrak{m}$.
Since $\mathfrak{m}$ is an ideal it has rank $1$, therefore also all its symmetric powers have rank $1$, and the same holds for the reflexive hull.
In other words, we have $\rank_R\left(\Sym^q_R(\Syz^1_R(\kk)\right)^{**}=1$ for all $q\in\mathbb{N}$.
Therefore, for each $q$ we have only two possibilities: either $\mathcal{S}^q\cong R$ and $\freerank_R\mathcal{S}^q=\rank_R\mathcal{S}^q=1$ or $\mathcal{S}^q$ is not free and  $\freerank_R\mathcal{S}^q=0$.
\end{Rem}

\par Now assume in addition that the category $\Ref{R}$ of finitely generated reflexive modules has the KRS property. 
We fix an indecomposable object $M$ in $\Ref{R}$, then for every $q\in\mathbb{N}$ we have a unique decomposition
\begin{equation*}
 \mathcal{S}^q=M^{a_q}\oplus N_q,
\end{equation*}
where the module $N_q$ contains no copy of $M$ as direct summand. 
The natural number $a_q$ is called \emph {the multiplicity of $M$ in $\mathcal{S}^q$}.

\begin{Def}\label{defgeneralizedsymmetric}
 The number 
 \begin{equation*}
  s_{\sigma}(R,M):=\lim_{N\rightarrow+\infty}\frac{\sum_{q=0}^Na_q}{\sum_{q=0}^N\rank_R\mathcal{S}^q}
 \end{equation*}
is called \emph{generalized symmetric signature of $R$ with respect to $M$}, provided the limit exists.
\end{Def}

\par We can define the symmetric signature also for $\mathbb{N}$-graded rings.
\par Let $R$ be a standard graded Noetherian $\kk$-domain of dimension $d$. 
We work in the category $\mmodgr{R}$ of finitely generated graded $R$-modules.
In this category we consider the graded module $\Syz^d_R(\kk)$ coming from a minimal graded free resolution of $\kk$ as $R$-module, and the differential module $\Omega_{R/\kk}$, which is also naturally graded.
Their $q$-th reflexive symmetric powers $\mathcal{S}^q$ and $\mathcal{C}^q$ are again graded modules. 
We also introduce a graded version of the free rank.
For any finitely generated graded $R$-module $M$, we define the \emph{graded free rank} of $M$ as
  \begin{equation*}
  \begin{split}
 \freerank_R^{\mathrm{gr}}(M):=\max\{n: \ \exists &\text{ a homogeneous of degree $0$ split surjection }\varphi: M\twoheadrightarrow F, \\ &\text{ with } F \text{ free graded $R$-module of rank } n\}. 
 \end{split}
 \end{equation*}
 
\par The main reason to introduce this definition is that we want work in the Krull-Schmidt category $\mmodgr{R}$ of finitely generated graded $R$-modules, whose maps are $R$-linear homomorphisms of degree $0$. 
Thus, the previous definition is more natural in this setting.

\begin{Def}\label{defgradedsymmetric}
The \emph{graded (syzygy) symmetric signature} of $R$ is defined by the limit
\begin{equation*}
  s_{\sigma}(R):=\lim_{N\rightarrow+\infty}\frac{\sum_{q=0}^N\freerank^{\mathrm{gr}}_R\mathcal{S}^q}{\sum_{q=0}^N\rank_R\mathcal{S}^q},
 \end{equation*}
provided it exists.
\end{Def}
Replacing the module $\mathcal{S}^q$ with $\mathcal{C}^q$ in the definition above, we define also the graded version of the differential symmetric signature.

\begin{Rem}
 Notice that in the definition of $s_{\sigma}(R)$ for graded rings we consider the  graded free rank instead of the free rank as in the local setting.
 Therefore some differences may occur.
 For example, one should remind that all direct summands of the form $R(a)$ for some integer $a$ contribute to the graded free rank part of the module. 
\end{Rem}

\section{Auslander correspondence and reflexive symmetric powers}\label{sectionAuslandersymmetric}
We fix the setting for this section.
Let $S=\kk\llbracket u,v\rrbracket$ be a power series ring in two variables over an algebraically closed field $\kk$.
Let $G\subseteq\Gl{2,\kk}$ be a finite group which is small, i.e. it contains no pseudo-reflections, and such that the characteristic of $\kk$ and the order $|G|$ of $G$ are coprime.
The group $G$ acts on $S$ via linear changes of variables, and we denote by $R:=S^G$ the invariant ring under this action.
The invariant ring $R$ is called \emph{quotient singularity} and has the following properties: it is a Noetherian local domain of dimension $2$, normal, Cohen-Macaulay, complete, and an isolated singularity (see e.g. \cite[Theorem 4.1]{BD08}).
Moreover, Watanabe \cite{Wat74a, Wat74b} proved that $R$ is Gorenstein if and only if $G\subseteq\Sl{2,\kk}$.
In this case $R$ is called \emph{special quotient singularity}.
If $G$ is a cyclic group $R$ is called \emph{cyclic quotient singularity}.
The \emph{Kleinian singularities}, or \emph{ADE singularities} are special quotient singularities, which actually coincides with them if the base field has characteristic zero.

\par In Section \ref{sectionKleinian} we will compute the symmetric signature of Kleinian singularities.
In this section we will review the theory of Auslander correspondence between linear $\kk$-representations of $G$ and MCM $R$-modules, and we will prove that the Auslander functor commutes with reflexive symmetric powers.
For a more detailed survey of these facts the reader may consult the books of Yoshino \cite{Yos90}, and of Leuschke and Wiegand \cite{LW12}, or the second author's Ph.D. thesis \cite{Cam16}.

\subsection{The Auslander correspondence}
\par We denote by $\kk[G]$ the \emph{group ring} of $G$ over $\kk$, that is the $\kk$-vector space with basis $\{e_g: g\in G\}$, and with multiplication given on the basis elements by group multiplication $e_ge_h=e_{gh}$ and extended linearly to arbitrary elements.
The category $\mmod{\kk[G]}$ of finitely generated left modules over the group ring can be identified with the category of linear $\kk$-representations of $G$.
For this reason we will use indifferently the terminology of representation theory or the terminology of module theory over $\kk[G]$.

\par The \emph{skew group ring} of $G$ and $S$ is denoted by $S*G$ and it is a free $S$-module on the elements of $G$ $S*G=\displaystyle\bigoplus_{g\in G}S g$, with a multiplicative structure given by $(s g)(t h):=sg(t)\cdot gh$ for all $s,t\in S$, $g,h\in G$.

\par The group ring and the skew group ring are in general non-commutative, so by $\kk[G]$-module and $S*G$-module we will always mean left module.
We point out that an $S*G$-module is just an $S$-module with a compatible action of $G$: $g(sm)=g(s)g(m)$ for all $g\in G$, $s\in S$ and $m\in M$. 
The ring $S$ is clearly an $S*G$-module, but not every $S$-module has a natural $S*G$-module structure.

\par There is a natural functor $\mathcal{F}$ between the category $\mmod{\kk[G]}$ and the category $\proj{S*G}$ of finitely generated projective $S*G$-modules, namely $\mathcal{F}(V):=S\otimes_{\kk}V$ for every object $V$ in $\mmod{\kk[G]}$.
The functor $\mathcal{F}$ has a right adjoint given by $\mathcal{F}'(P):=P\otimes_{S}\kk$.
The functors $\mathcal{F}$ and $\mathcal{F}'$ are an adjoint pair \cite[Lemma 10.1]{Yos90}, but in general they are not an equivalence of categories.

\par We remark that an $S*G$-module is $S*G$-projective if and only if it is $S$-projective, hence $S$-free, since $S$ is local.

\par We define two other functors $\mathcal{G}:\proj{S*G}\rightarrow\Ref{R}$, $\mathcal{G}(M)=M^G$, and $\mathcal{G}':\Ref{R}\rightarrow\proj{S*G}$, $\mathcal{G}'(N)=(S\otimes_RN)^{**}$, where $(-)^{*}:=\Hom_S(-,S)$, and $\Ref{R}$ denotes the category of finitely generated reflexive $R$-modules.
Since $R$ has dimension $2$ the category $\Ref{R}$ coincides with the category $\cmod{R}$ of maximal Cohen-Macaulay $R$-modules, and also with the category $\mathrm{Add}_R(S)$ of $R$-direct summands of $S$ by a result of Herzog \cite{Her78b}, that is $\Ref{R}=\cmod{R}=\mathrm{Add}_R(S)$.

\par Auslander \cite{Aus86b} proved that the functors $\mathcal{G}$ and $\mathcal{G}'$ are an equivalence of categories.
We will give a geometric proof of this fact, based on the following ring version of the Speiser's Lemma of Galois theory due to Auslander and Goldman \cite{AG60}, and Chase, Harrison, and Rosenberg \cite{CHR65}.

\begin{Lemma}[Speiser's Lemma] \label{speiserlemma}
Let $A$ be a commutative domain, $G$ a finite group of ring automorphisms of $A$ and denote by $B=A^G$ the invariant ring. 
Then the following facts are equivalent.
\begin{compactenum}[1)]
 \item $A$ is a separable $B$-algebra.
 \item For every $g\neq \mathrm{id}_G$ in $G$, and every maximal ideal $\mathfrak{p}$ in $A$, there exists $a\in A$ such that $a-g(a)\not\in\mathfrak{p}$.
 \item The map $\varphi:A\otimes_BM^G\rightarrow M$ given by $\varphi(a\otimes m)=am$ is an isomorphism for every $A*G$-module $M$.
\end{compactenum}
 \end{Lemma}

\begin{Lemma}\label{M=SxMGlemmadim2}
Let $S=\kk\llbracket u,v\rrbracket$, let $G\subseteq\Gl{2,\kk}$ be a finite small group such that $(\chara\kk,|G|)=1$, and let $R=S^G$.
 Let $M$ be a projective $S*G$-module and consider the canonical map
 \begin{equation*}
 \delta: S\otimes_R M^G\rightarrow M,
 \end{equation*}
 given by $\delta(s\otimes m)=sm$. Then the following facts hold. 
 \begin{compactenum}[1)]
  \item $\delta$ induces an isomorphism of coherent sheaves $\reallywidetilde{S\otimes_R M^G}\rightarrow\widetilde{M}$ on the punctured spectrum $U'=\mathrm{Spec} \,S\setminus\{\mathrm{m}\}$, where $\mathfrak{m}=(u,v)$.
  \item $\delta$ induces an isomorphism of $S*G$-modules 
  \begin{equation*}
   (S\otimes_R M^G)^{**}\cong M.
 \end{equation*}
 \end{compactenum}
\end{Lemma}

\begin{proof}
The ring $S$ is a Cohen-Macaulay isolated singularity, so coherent sheaves associated to MCM $S$-modules are locally free on the punctured spectrum. 
Moreover we recall that projective $S$-modules are MCM in this case. 
Therefore it follows that property \textit{1)} implies property \textit{2)}.
\par We prove \textit{1)}.  Let $f_1,\dots,f_{\mu}$ be elements in $R$ which generate $\mathfrak{m}$ up to radical, that is $\sqrt{(f_1,\dots,f_{\mu})}=\mathfrak{m}$ as ideals in $S$. 
These elements exist, since the ring extension $R\hookrightarrow S$ is finite. 
Thus, we have an open covering $U'=\bigcup_{i}D(f_i)$.
\par We claim that for every $f=f_i$ the induced map 
\begin{equation*}
  \delta_f: S_f\otimes_{R_f}M_f^G\rightarrow M_f
 \end{equation*}
is an isomorphism. We prove the claim, then the Lemma will follow from it. 
In fact, because we have a global homomorphism, we get a sheaf homomorphism defined on $U'$ which is locally an isomorphism, so it is forced to be an isomorphism on $U'$.
\par To prove the claim, let $f$ be one of the $f_i$'s.
We check that condition \textit{2)} of Lemma \ref{speiserlemma} is true for $A=S_f$ and $B=R_f$.
Let $\mathfrak{p}$ be a maximal ideal of $A$.
If $a-g(a)\in\mathfrak{p}$ for every $a\in A$ and every $g\in G$, then we have $a\in\mathfrak{p}$ if and only if $g(a)\in\mathfrak{p}$, which is equivalent to say that $\mathfrak{p}$ is a fix point for the action of $G$ on $A$.
On the other hand, since $G$ is small, its action on $U'$ is fixpoint-free, so we get a contradiction.
Therefore $\delta_f$ is an isomorphism by Lemma \ref{speiserlemma} above.
\end{proof}

 \begin{Theorem}\label{RandSGmodules}
 The functors $\mathcal{G}:\proj{S*G}\rightarrow\Ref{R}$ and $\mathcal{G}':\Ref{R}\rightarrow\proj{S*G}$ give an equivalence of categories
\begin{equation*}
 \proj{S*G}\cong\Ref{R}.
\end{equation*}
\end{Theorem}
\begin{proof}
See \cite[Proposition 10.9]{Yos90}.
\end{proof}

\par We compose the functors $\mathcal{F}$ and $\mathcal{G}$ and we obtain a functor $\mathcal{A}:\mmod{\kk[G]}\rightarrow\Ref{R}$, $\mathcal{A}(V):=(S\otimes_{\kk}V)^G$ called \emph{Auslander functor}.
The Auslander functor has a right adjoint given by $\mathcal{A}'(N):=(S\otimes_RN)^{**}\otimes_{S}\kk$, for every reflexive $R$-module $N$.
These functors give a one-one correspondence between $\kk$-representations of $G$ and maximal Cohen-Macaulay modules over $R$ called \emph{Auslander correspondence}.
We collect the properties of the functor $\mathcal{A}$ and of the Auslander correspondence in the following theorem and corollary.

\begin{Theorem}[Auslander correspondence]\label{Auslandertheorem}
  The functors $\mathcal{A}:\mmod{\kk[G]}\rightarrow\Ref{R}$ and  $\mathcal{A}':\Ref{R}\rightarrow\mmod{\kk[G]}$ have the following properties.
 \begin{compactenum}[1)]
  \item $\mathcal{A}(V)\cong\mathcal{A}(W)$ if and only if $V\cong W$.
  \item $\mathcal{A}(V)$ is indecomposable in $\Ref{R}$ if and only $V$ is an irreducible representation.
  \item $\mathcal{A}(\Hom_{\kk}(V,W))\cong\Hom_R(\mathcal{A}(V),\mathcal{A}(W))$ for every $V,W\in\mmod{\kk[G]}$.
  \item $\mathcal{A}(V\otimes_{\kk}W)\cong\left(\mathcal{A}(V)\otimes_R\mathcal{A}(W)\right)^{**}$ for every $V,W\in\mmod{\kk[G]}$.
  \item If $V_0$ is the trivial representation then $\mathcal{A}(V_0)=R$.
  \item $\rank_R\mathcal{A}(V)=\dim_{\kk}V$ for every $\kk[G]$-module $V$.
 \end{compactenum}
\end{Theorem}

\begin{Cor}\label{corollarydecomposition}
Let $V_1,\dots,V_r$ be a complete set of non-isomorphic irreducible $\kk$-representations of $G$, and fix $N_i=\mathcal{A}(V_i)=(S\otimes_{\kk}V_i)^G$. 
Then $N_1,\dots,N_r$ is a complete set of non-isomorphic indecomposable MCM $R$-modules. 
Moreover let $V$ be a $\kk$-representation which decomposes as
 \begin{equation*}
  V=V_1^{n_1}\oplus\cdots\oplus V_r^{n_r},
 \end{equation*}
for some natural numbers $n_i$. 
Then the MCM $R$-module $N:=\mathcal{A}(V)=(S\otimes_{\kk}V)^G$ decomposes as
 \begin{equation*}
  N=N_1^{n_1}\oplus\cdots\oplus N_r^{n_r}.
 \end{equation*}
\end{Cor}

\begin{Ex}\label{examplecyclicgroup3}
 Let $C_n=<g>$ be the cyclic group of order $n$, with irreducible representations $(V_t,\rho_t)$ over an algebraically closed field $\kk$ given by $\rho_t(g)=\xi^t$, for a fixed primitive $n$-th root of unity $\xi\in\kk$. 
 Assume that $\chara\kk$ does not divide $|G|$. 
 We can embed $G$ into $\Gl{2,\kk}$ via the representation $V_1\oplus V_a$, where $a$ is a natural number such that $(a,n)=1$, otherwise the representation is not faithful. 
 In other words, we consider the cyclic group generated by
 \begin{equation*}
  \begin{pmatrix}
   \xi&0\\0&\xi^a
  \end{pmatrix}.
 \end{equation*}
This group acts linearly on $S=\kk\llbracket u,v\rrbracket$ and the invariant subring $R$ is generated by monomials $u^iv^j$ such that $i+aj\cong0\mod n$.
\par For each irreducible representation $V_t$, we have an indecomposable MCM $R$-module $M_t=(S\otimes_{\kk}V_t)^G$. A straightforward computation shows that this is given by
\begin{equation*}
 M_t=R\left(u^iv^j: \ i+aj\cong-t\mod n\right).
\end{equation*}
\end{Ex}

\subsection{Auslander functor and reflexive symmetric powers}
 \par We want to prove that the Auslander functor commutes with $q$-th reflexive symmetric powers $\Sym^q(-)^{**}$, that is
 \begin{equation*}
  \mathcal{A}(\Sym_{\kk}^q(V))\cong(\Sym_R^q(\mathcal{A}(V)))^{**}
 \end{equation*}
for every $\kk[G]$-module $V$. 
Since the Auslander functor is the composition of the functors $\mathcal{F}$ and $\mathcal{G}$ we will split the proof of this fact in two propositions.

 \begin{Rem}
 Observe that in the categories $\mmod{\kk[G]}$ and $\proj{S*G}$ the reflexive symmetric powers coincide with the usual symmetric powers. 
 In fact, every finitely generated $\kk[G]$-module $V$ is reflexive, so it is canonically isomorphic to its double dual $V^{**}$. 
 Since $S$ is regular, for every finitely generated projective $S*G$-module $N$ the symmetric powers $\Sym^q_S(N)$ are $S$-free, hence reflexive. 
 \end{Rem}

\begin{Prop}\label{symcommuteswithauslanderprop1}
For every $\kk[G]$-module $V$ we have an isomorphism of $S*G$-modules
\begin{equation*}
 \Sym_S^q(S\otimes_{\kk}V)\cong S\otimes_{\kk}\Sym^q_{\kk}(V).
\end{equation*}
\end{Prop}

\begin{proof}
\par From \cite[(6.5.1)]{Nor84} we have an isomorphism of $S$-modules $\psi:\Sym_S^q(S\otimes_{\kk}V)\rightarrow S\otimes_{\kk}\Sym^q_{\kk}(V)$ given by $\psi((a_1\otimes v_1)\circ\dots\circ (a_q\otimes v_q))=(a_1\cdots a_q)\otimes(v_1\circ\dots\circ v_q)$.
The only thing that we have to check is that $\psi$ is compatible with the action of $G$, and this is shown by the following commutative diagram 
\begin{equation*}
 \begin{tikzcd}
  &(a_1\otimes v_1)\circ\dots\circ (a_q\otimes v_q) \arrow{r}{\psi}\arrow{d}{g}
  &(a_1\cdots a_q)\otimes(v_1\circ\dots\circ v_q) \arrow{d}{g}\\
  &\left(g(a_1)\otimes g(v_1)\right)\circ\dots\circ\left(g(a_q)\otimes g(v_q)\right)  \arrow{r}{\psi}
  &\begin{split}&g(a_1\dots a_q)\otimes g(v_1\circ\dots\circ v_q)= \\ &\left(g(a_1)\cdots g(a_q)\right)\otimes \left(g(v_1)\circ\dots\circ g(v_q)\right).\end{split}
 \end{tikzcd}
\end{equation*}
\phantom{some  text to complete some  text to complete some  text to complete some  more text }
\end{proof}

\par We will use often the following well-known lemma.
We write it here for ease of reference.

\begin{Lemma}\label{reflexivehullpuncturedlemma}
 Let $(R,\mathfrak{m},\kk)$ be a normal two-dimensional local domain, and let $U=\mathrm{Spec} \,R\setminus\{\mathfrak{m}\}$ be the punctured spectrum of $R$.
 For a torsion-free finitely generated $R$-module $M$ we have an isomorphism
 \begin{equation*}
  M^{**}\cong\Gamma(U,\widetilde{M}).
 \end{equation*}
\end{Lemma}

\begin{Rem}\label{identificationremark}
Let $S=\kk\llbracket u,v\rrbracket$ and $R=S^G$ be as in the rest of the section.
Consider the following commutative diagram 
\begin{equation*}
 \begin{CD}
  U'@>>>\mathrm{Spec} \,S \\
  @V\pi VV @V\pi VV \\
  U@>>>\mathrm{Spec} \,R
 \end{CD}
\end{equation*}
where the map $\pi$ is induced by the inclusion $R\hookrightarrow S$, $U$ is the punctured spectrum of $R$, and $U'=\pi^{-1}(U)$ is the pull-back of $U$ to $\mathrm{Spec} \,S$, which is actually the punctured spectrum of $S$. 
For every $R$-module $M$ this gives us the following identification of sheaves on $U'$
\begin{equation}\label{identification}
 \pi^{*}(\widetilde{M}|_{U})\cong\left.\reallywidetilde{S\otimes_RM}\right|_{U'}.
\end{equation}
\end{Rem}

\begin{Prop}\label{symcommuteswithauslanderprop2}
 Let $N$ be a projective $S*G$-module, then we have an isomorphism of $R$-modules
 \begin{equation*}
(\Sym_S^q(N))^G\cong\left(\Sym_R^q(N^G)\right)^{**}.  
 \end{equation*}
\end{Prop}

\begin{proof}
Let $\mathcal{G}'(-)=(S\otimes_R-)^{**}$ be the right adjoint of the $G$-invariants functor $\mathcal{G}(-)=(-)^G$. 
Since $\mathcal{G}$ and $\mathcal{G}'$ are an equivalence of categories (Theorem \ref{RandSGmodules}), it is enough to show that
\begin{equation*}
\mathcal{G}'\left((\Sym_S^q(N))^G\right)\cong\mathcal{G}'\left(\left(\Sym_R^q(N^G)\right)^{**}\right),  
 \end{equation*}
that is 
\begin{equation}\label{aussymequat1}
\left(S\otimes_R\left(\Sym^q_S(N)\right)^G\right)^{**}\cong\left(S\otimes_R\left(\Sym_R^q(N^G)\right)^{**}\right)^{**},
 \end{equation}
 as $S*G$-modules.
\par Notice that the two double duals on the right hand side of \eqref{aussymequat1} are different: the first one is the double dual in $\mmod{R}$ and the second is the double dual in $\mmod{S}$. 
\par In order to prove \eqref{aussymequat1}, we consider the left hand side and the right hand side separately. 
From the second part of Lemma \ref{M=SxMGlemmadim2}, the left hand side of \eqref{aussymequat1} is $\left(S\otimes_R\left(\Sym^q_S(N)\right)^G\right)^{**}\cong\Sym^q_S(N)$.
\par For the $S$-module on the right hand side of \eqref{aussymequat1} we use  Lemma \ref{reflexivehullpuncturedlemma}, and we interpret it as the evaluation of the sheaf 
\begin{equation*}
\left.\reallywidetilde{S\otimes_R\left(\Sym_R^q(N^G)\right)^{**}}\right|_{U'}
\end{equation*}
on the punctured spectrum $U'$ of $S$.
 From the commutative diagram of Remark \ref{identificationremark} we get the isomorphism
\begin{equation*}
 \left.\reallywidetilde{S\otimes_R\left(\Sym^q_R(N^G)\right)^{**}}\right|_{U'}\cong\pi^{*}\left(\left.\reallywidetilde{\Sym_R^q(N^G)^{**}}\right|_{U}\right)=\pi^{*}\left(\left.\reallywidetilde{\Sym_R^q(N^G)}\right|_{U}\right),
\end{equation*}
where we can remove the double dual over $U$, thanks to Lemma \ref{reflexivehullpuncturedlemma}.
 Since taking symmetric powers commute with sheafification and with the restriction map of sheaves we get 
\begin{equation*}
 \pi^{*}\left(\left.\reallywidetilde{\Sym_R^q(N^G)}\right|_{U}\right)\cong \pi^{*}\left(\left.\left(\Sym^q_{X}(\widetilde{N^G})\right)\right|_{U}\right)\cong \pi^{*}\left(\Sym^q_{X}\big(\widetilde{N^G}|_{U}\big)\right),
\end{equation*}
where the second and the third symmetric powers are sheaf symmetric powers taken over $X=\mathrm{Spec} \,R$. 
We set $Y=\mathrm{Spec} \,S$. Since sheaf symmetric powers commute with pullback we have
\begin{equation*}
\pi^{*}\left(\Sym^q_{X}\big(\widetilde{N^G}|_{U}\big)\right)\cong\Sym_{Y}^q\left(\pi^{*}\big(\widetilde{N^G}|_{U}\big)\right).
\end{equation*}
We apply again \eqref{identification} and Lemma \ref{reflexivehullpuncturedlemma} to obtain 
\begin{equation*}
\Sym_{Y}^q\left(\pi^{*}\big(\widetilde{N^G}|_{U}\big)\right)\cong\Sym_{Y}^q\left(\left.\reallywidetilde{S\otimes_RN^G}\right|_{U'}\right)=\Sym_{Y}^q\left(\left.\reallywidetilde{(S\otimes_RN^G)^{**}}\right|_{U'}\right).
\end{equation*}
Since taking symmetric powers commutes with sheafification and with the restriction map of sheaves we get 
\begin{equation*}
 \Sym_{Y}^q\left(\left.\reallywidetilde{(S\otimes_RN^G)^{**}}\right|_{U'}\right)\cong \left.\reallywidetilde{\Sym_S^q\left(\left(S\otimes_RN^G\right)^{**}\right)}\right|_{U'}\cong\left.\reallywidetilde{\Sym_S^q\left(N\right)}\right|_{U'},
\end{equation*}
where the last isomorphism follows from Lemma \ref{M=SxMGlemmadim2}. 
Taking global sections on $U'$ we obtain that the right hand side of \eqref{aussymequat1} is also isomorphic to  $\Sym^q_S(N)$.
 Therefore we have an isomorphism of $S$-modules as in \eqref{aussymequat1}. 
 This is actually an isomorphism of $S*G$-modules.
In fact, the natural map of $R$-modules $\Sym^q_R(N^G)\rightarrow\left(\Sym_S^q(N)\right)^G$ induces a natural map  $S\otimes_R\Sym^q_R(N^G)\rightarrow S\otimes_R\left(\Sym_S^q(N)\right)^G$, which is $G$-compatible since the action is just on $S$.
 \end{proof}

From Proposition \ref{symcommuteswithauslanderprop1} and Proposition \ref{symcommuteswithauslanderprop2} we immediately obtain the following.

 \begin{Theorem}\label{symcommuteswithauslandertheorem}
  Let $V$ be a $\kk[G]$-module, and let $M=(S\otimes_{\kk}V)^G$ be the corresponding MCM $R$-module via Auslander functor. Then we have
  \begin{equation*}
   \Sym_R^q(M)^{**}\cong \left(S\otimes_{\kk}\Sym_{\kk}^q(V)\right)^G.
 \end{equation*}
 In other words $\Sym_R^{q}(\mathcal{A}(V))^{**}\cong\mathcal{A}(\Sym^{q}_{\kk}(V))$.
 \end{Theorem}

\section{Symmetric signature of Kleinian singularities}\label{sectionKleinian}
\par In this section we will compute the generalized symmetric signature (and differential symmetric signature) of two-dimensional Kleinian singularities.
These are invariant rings $R:=\kk\llbracket u,v\rrbracket^G$, where $\kk$ is an algebraically closed field, and $G$ is one of the following finite small subgroups of $\Sl{2,\kk}$: the cyclic group $C_n$ of order $n$, the binary dihedral group $BD_n$ of order $4n$, the binary tetrahedral group $BT$ of order $24$, the binary octahedral group $BO$ of order $48$, and the binary icosahedral group $BI$ of order $120$.
We always assume that the characteristic of $\kk$ does not divide the order of the group $G$.
It is an old result of Klein that if $\kk$ has characteristic zero, then every finite subgroup of $\Sl{2,\kk}$ is isomorphic to one of these groups.
For this reason these groups are also called \emph{Klein groups}.
\par The Kleinian singularities are hypersurface rings, that is they are isomorphic to a quotient ring $\kk\llbracket x,y,z\rrbracket/(f)$ for some polynomial $f$, according to the following table.
\begin{center}\label{table}
\begin{tabular}{c|c|c|c}
singularity name & $G$ & $|G|$ & $f$ \\
\midrule
$A_{n-1}$ & cyclic & $n$ & $y^{n}-xz$\\ 
$D_{n+2}$ & binary dihedral  & $4n$ & $x^2+y^{n+1}+yz^2$ \\
$E_6$ & binary tetrahedral & $24$ & $x^2+y^3+z^4$\\
$E_7$ & binary octahedral & $48$ & $x^2+y^3+yz^3$ \\
$E_8$ & binary icosahedral & $120$ & $x^2+y^3+z^5$
\end{tabular}
\end{center}

\subsection{The fundamental module of Kleinian singularities}
\par Let $(R,\mathfrak{m},\kk)$ be a complete two-dimensional normal non-regular domain with canonical module $K_R$.
In this setting Auslander \cite{Aus86b} proved  that in the category $\Ref{R}$ there exists a unique non-split short exact sequence of the form
\begin{equation*}
 0\rightarrow K_R\rightarrow E\rightarrow \mathfrak{m}\rightarrow 0,
\end{equation*}
which is called the \emph{fundamental sequence} of $R$. 

\par The module $E$ appearing in the middle term is also unique up to isomorphism and is called the \emph{fundamental module} or \emph{Auslander module} of $R$ (cf. \cite[Chapter 11]{Yos90}). 
It is a reflexive (hence MCM) module of rank $2$. 
Moreover we have an isomorphism of $R$-modules $\left(\bigwedge^2 E\right)^{**}\cong K_R$.
\par The following example (cf. \cite[Example 11.8]{Yos90}) clarifies the name fundamental module.

\begin{Ex}\label{fundamentalmodulerepresentation}
Let $V$ be a $\kk$-vector space of dimension $2$ with basis $u$, $v$ and let $G\subseteq\Gl{2,\kk}$ be a finite subgroup. 
Then, $G$ acts on the power series ring $S=\kk\llbracket u,v\rrbracket$ and we consider the invariant ring $R=S^G$, which is a two-dimensional quotient singularity.
The fundamental module of $R$ is the image via Auslander functor of the fundamental representation $V$ of $G$, that is
\begin{equation*}
 E=(S\otimes_{\kk}V)^G.
\end{equation*}
\end{Ex}

\begin{Theorem}[Yoshino-Kawamoto, \cite{YK88}]\label{yoshinokawamoto1}
Let $R=\kk\llbracket x,y,z\rrbracket/(f)$ be a complete two-dimensional normal non-regular domain with canonical module, with  $f\in\kk\llbracket x,y,z\rrbracket$.
Then the fundamental module $E$ is isomorphic to the third syzygy of $\kk$, i.e. 
\begin{equation*}
E\cong\mathrm{Syz}^3_R(\kk).
\end{equation*}
Moreover, the following facts are equivalent.
\begin{compactenum}
\item The fundamental module $E$ is decomposable.
\item $R$ is a cyclic quotient singularity.
\end{compactenum}
\end{Theorem}

\par Observe that the Kleinian singularities satisfy the condition of the previous Theorem \ref{yoshinokawamoto1}.
Now we prove that also the second syzygy of the residue field is isomorphic to the fundamental module.

\begin{Theorem}\label{Theoremsecondsyzygyfield1}
Let $(R,\mathfrak{m},\kk)$ be a two-dimensional Kleinian singularity. 
Then the second syzygy of the residue field $\kk$ is isomorphic to the fundamental module, that is
\begin{equation*}
\Syz_R^2(\kk)\cong E.
\end{equation*}
\end{Theorem}

\par From Example \ref{fundamentalmodulerepresentation}, it is clear that Theorem \ref{Theoremsecondsyzygyfield1} is equivalent to the following statement.

\begin{Theorem}\label{Theoremsecondsyzygyfield2}
Let $(R,\mathfrak{m}_R,\kk)$ be a two-dimensional Kleinian singularity, and let $V_1$ be the two-dimensional fundamental representation of the acting group $G$ of $R$.
Then, the second syzygy of the residue field $\kk$ is isomorphic to the image of $V_1$ via Auslander functor, that is
\begin{equation*}
\Syz_R^2(\kk)\cong\mathcal{A}(V_1) 
\end{equation*}
\end{Theorem}

\par Theorem \ref{Theoremsecondsyzygyfield2} can be proved by case considerations. 
We illustrate the general strategy and we apply it to the singularity $A_{n-1}$ in Example \ref{examplecyclicsyzygy}. 
The reader interested to the computations for $D_{n+2}$, $E_6$, $E_7$ and $E_8$ may consult \cite{Cam16}.

\begin{proof}
\par Let $S=\kk\llbracket u,v\rrbracket$ be the powers series ring over an algebraically closed field $\kk$, and let $G$ be one of the Klein subgroups of $\Sl{2,\kk}$ acting on $S$ through a faithful representation $V_1$. 
Assume that $\chara\kk$ and $|G|$ are coprime. Then we have $R=S^G$, and  we denote by $\mathfrak{m}_R$ the maximal ideal of $R$. 
Let $p_1, p_2, p_3$ be a minimal system of generators of $\mathfrak{m}_R$ as ideal in $S$. 

\par We consider the following short exact sequence of $S$-modules, which is the beginning of an $S$-free resolution of $ S/\mathfrak{m}_RS$
\begin{equation}\label{syzygysequenceKlenian}
0\rightarrow \Syz^1_S(p_1,p_2,p_3)\rightarrow S^3 \xrightarrow{p_1,p_2,p_3} S\rightarrow S/\mathfrak{m}_RS\rightarrow0.
\end{equation}
The group $G$ acts linearly on $S$ through the fundamental representation $V_1$, and this action extends naturally to $S^3$, and its submodule $\Syz^1_S(p_1,p_2,p_3)$.
In other words, the sequence \eqref{syzygysequenceKlenian} is an exact sequence of $S*G$-modules.
We apply the functor $\mathcal{G}(N)=N^G$, which is exact, to the sequence \eqref{syzygysequenceKlenian} and we obtain an exact sequence
\begin{equation*}
0\rightarrow M\rightarrow R^3 \rightarrow R\rightarrow R/\mathfrak{m}_R\rightarrow0.
\end{equation*}
\par Since $p_1,p_2,p_3$ are a system of generators for the maximal ideal $\mathfrak{m}_R$ of $R$, the last non-zero module on the right is the residue field.
It follows that the module $M$ appearing on the left of the last sequence is just the second syzygy of $\kk$, that is $M=\Syz_R^2(\kk)$.
In other words we have that
\begin{equation*}
\Syz^1_S(p_1,p_2,p_3)^G=\Syz_R^2(\kk).
\end{equation*}
So, to understand which $R$-module $\Syz_R^2(\kk)$ is, we need to understand which is the action of $G$ on  $\Syz^1_S(p_1,p_2,p_3)$, that is its $S*G$-module structure.

\par We know that as $S$-module $\Syz_S^1(p_1,p_2,p_3)\cong S^2$ and it is therefore generated by two elements $s_1(u,v),s_2(u,v)\in S^3$.
We need to keep track of the action of $G$ through this isomorphism.
The action of $G$ on $\Syz^1_S(p_1,p_2,p_3)$ is inherited by the action on $S$, which is linear and given by matrices $M=M_g$ in $\Sl{2,\kk}$.
\par In order to understand how these matrices act on the generators $s_1(u,v)$ and $s_2(u,v)$ we procede as follows.
For each matrix $M$ we apply the linear transformation $(u,v)\mapsto M (u,v)^T $ to $s_1(u,v)$ and $s_2(u,v)$. 
We obtain two elements $s'_1(u,v)$ and $s'_2(u,v)$ in $S^3$ which belong to $\Syz^1_S(p_1,p_2,p_3)$. 
Therefore we can write them as linear combination of $s_1$ and $s_2$
\begin{equation*}
\begin{pmatrix} s'_1\\
s'_2
\end{pmatrix}
= N \begin{pmatrix} s_1 \\ s_2
\end{pmatrix}
\end{equation*}
for some matrix $N$. 
In this way we obtain a collection of matrices $N=N_g$ for $g\in G$, which gives us the representation corresponding to $\Syz^1_S(p_1,p_2,p_3)$.
Then one has to check that this is actually isomorphic to the fundamental representation $V_1$.
\end{proof}

\begin{Ex}\label{examplecyclicsyzygy}
  Let $\xi$ be a primitive $n$-th root of unity in $\kk$, and consider the cyclic group $C_n$ generated by
\begin{equation*}
A=\begin{pmatrix}
\xi & 0 \\ 0 & \xi^{-1}
\end{pmatrix}.
\end{equation*}
The maximal ideal $\mathfrak{m}_R$ of the invariant ring $R=\kk\llbracket u,v\rrbracket^{C_n}$ is generated by polynomials $p_1=u^n$, $p_2=v^n$, $p_3=uv$. Their syzygy module $\Syz^1_S(p_1,p_2,p_3)$ is generated by 
\begin{equation*}
s_1=\begin{pmatrix}
0\\-u\\v^{n-1}
\end{pmatrix} \ \ \text{ and } \ \ 
s_2=\begin{pmatrix}
-v \\ 0 \\ u^{n-1}
\end{pmatrix}.
\end{equation*}
We apply the linear transformation given by $A$: $u\mapsto\xi u$, $v\mapsto \xi^{-1}v$, to $s_1$ and $s_2$
\begin{equation*}\begin{split}
s_1&\mapsto s'_1=\begin{pmatrix}
0\\-\xi u\\(\xi^{-1})^{n-1} v^{n-1}
\end{pmatrix}=\xi\begin{pmatrix}
0\\-u\\v^{n-1}
\end{pmatrix}=\xi s_1;\\
s_2&\mapsto s'_2=\begin{pmatrix}
-\xi^{-1}v \\ 0 \\ \xi^{n-1}u^{n-1}
\end{pmatrix}=\xi^{-1}\begin{pmatrix}
-v \\ 0 \\ u^{n-1}
\end{pmatrix}=\xi^{-1}s_2.
\end{split}
\end{equation*}
Thus, the representation corresponding to $\Syz^1_S(p_1,p_2,p_3)$ is exactly the fundamental representation.
\end{Ex}

\par The referee suggested to us an alternative proof of Theorem \ref{Theoremsecondsyzygyfield2} which works for the singularities $D_{n+2}$, $E_6$, $E_7$ and $E_8$.
We sketch it in the following remark.

\begin{Rem}\label{remarksecondsyzygyfield2}
We use the same notation as in the proof of Theorem \ref{Theoremsecondsyzygyfield2}.
First, we observe that by Auslander correspondence the isomorphism $\Syz_R^2(\kk)\cong\mathcal{A}(V_1)$ is equivalent to say that the top of $\mathrm{Ext}^2_S(S/\mathfrak{m}_RS,S)$ is $V_1^*$.  By the equivariant local duality (cf. \cite[(5.5)]{HO10}), this is equivalent to say that the socle of $S/\mathfrak{m}_RS$ is $V_1$.
Now let $x,y,z\in S$ be a minimal system of generators of $\mathfrak{m}_R$ as in the table at page \pageref{table}, and consider the quotient $\Lambda=S/(y,z)$.
Since $y$ and $z$ are invariants, $\Lambda$ inherits a natural $G$-action.
By the relation $f$ in the table, $\Lambda$ is Artinian and hence it is a complete intersection.
Moreover, the degree relation $\deg x= \deg y+\deg z-1$ holds and implies that $x\cdot u=x\cdot v=0$, because $\deg(x\cdot u)=\deg(x\cdot u)=\deg y+ \deg z$.
So $x$ must sit in the socle of $\Lambda$, which is one-dimensional and is therefore equal to $\Lambda_{\deg x}$, the degree $\deg x$ component of $\Lambda$.
Since $x$ is invariant, the socle of $\Lambda$ is the trivial representation. 
The natural multiplication $\Lambda_i\times\Lambda_{\deg x-i}\rightarrow\Lambda_{\deg x}\cong\kk$ forces $\Hom_{\kk}(\Lambda,\kk)\cong\Lambda$ as $S*G$-modules, and moreover we have isomorphisms in each degree, i.e. $\Lambda_i\cong\Hom_{\kk}(\Lambda,\kk)_{\deg x-i}$.
We quotient by $x$, and we have that the socle of $\Lambda/x\Lambda\cong S/\mathfrak{m}_RS$, which is two-dimensional, must agree with $\Lambda_{\deg x -1}$, which is also two-dimensional.
So we obtain $\Lambda_{\deg x-1}\cong \Hom_{\kk}(\Lambda,\kk)_{1}\cong \Hom_{\kk}(V_1,\kk)\cong V_1$, since the fundamental representation of the Klein groups is self-dual, and we are done.
\end{Rem}

\par For Kleinian singularities, also the module of Zariski differentials $\Omega_{R/\kk}^{**}$ is isomorphic to the fundamental module.

\begin{Theorem}[Martsinkovsky, \cite{Mar90}]\label{Zariskiisfundamental} 
Let $(R,\mathfrak{m},\kk)$ be a two-dimensional quotient singularity over an algebraically closed field $\kk$, and assume that the characteristic of $\kk$ does not divide the order of the acting group.
Then the module of Zariski differentials $\Omega_{R/\kk}^{**}$ of $R$ over $\kk$ is isomorphic to the fundamental module $E$. 
\end{Theorem}

\subsection{The decomposition of $\Sym^q(V)$}

\par Let $G$ be a finite group of order $n$ and let $\kk$ be an algebraically closed field such that $\chara\kk$ does not divide $n$.  
Let $\mu_n(\kk)$ be the group of $n$-th roots of unity in $\kk$ and let $\mu_n(\mathbb{C})$ be the group of complex $n$-th roots of unity. 
Both  $\mu_n(\kk)$ and  $\mu_n(\mathbb{C})$ are cyclic groups of order $n$, so we can fix an isomorphism $ \phi: \mu_n(\kk)\rightarrow\mu_n(\mathbb{C})$, which we name a \emph{lift}. 
In the same way, we say that a complex root of unity $z\in\mu_n(\mathbb{C})$ is a lift of $a\in\mu_n(\kk)$ if $z=\phi(a)$.
\par Let $(V,\rho)$ be a $\kk$-representation of $G$ of dimension $r\geq1$ and let $g$ be an element of $G$. 
Then the matrix $\rho(g)$ is diagonalizable in $\kk$ and its eigenvalues are elements of $\mu_n(\kk)$, since the order of $g$ divides $n$. 
Let $\lambda_1,\dots,\lambda_r$ be these eigenvalues, counted with multiplicity.
 The \emph{Brauer character} or simply the \emph{character} of $(V,\rho)$ is the function $\chi:G\rightarrow\mathbb{C}$ given by
 \begin{equation*}
 \chi_V(g)=\phi(\lambda_1)+\cdots+\phi(\lambda_r).
 \end{equation*}

\par This definition depend on the choice of the isomorphism $\phi$. 
Since there are in general many choices for $\phi$, we have a certain degree of arbitrariness. 
However once chosen $\phi$, it will never be changed and sometimes we will simply say that we lift the eigenvalues to $\mathbb{C}$, meaning that the isomorphism $\phi$ is fixed. 

\par For a proof of the following classical results of character theory the reader may consult the books of Feit \cite[Chapter IV]{Fei82}, and of Fulton and Harris \cite[Chapter 2]{FH91}.

\begin{Prop}\label{propertiesofcharacterprop}
 Let $V$ and $W$ be $\kk$-representations of $G$ with characters $\chi_V$ and $\chi_W$.
 Then the following facts hold:
 \begin{enumerate}
  \item $\chi_V(e)=\dim_{\kk}V$, where $e$ is the identity of $G$;
  \item $\chi_V(g^{-1})=\overline{\chi_V(g)}$, the complex conjugate of $\chi_V(g)$, for every $g\in G$;
  \item $\chi_V(hgh^{-1})=\chi_V(g)$ for every $g,h\in G$;
  \item $\chi_{V\oplus W}(g)=\chi_V(g)+\chi_W(g)$ for every $g\in G$;
  \item $\chi_{V\otimes W}(g)=\chi_V(g)\cdot\chi_W(g)$ for every $g\in G$.
 \end{enumerate} 
\end{Prop}

\par We define an Hermitian inner product on the set of characters of $G$ by
\begin{equation*}
 \langle\varphi,\psi\rangle:=\frac{1}{|G|}\sum_{g\in G}\overline{\varphi(g)}\psi(g),
\end{equation*}
for every $\varphi, \psi$ characters, where $\overline{\varphi(g)}$ denotes the complex conjugation.
This bilinear form satisfies
 \begin{equation*}
  \langle \chi,\psi\rangle=\frac{1}{|G|}\sum_{g\in G}\chi(g)\overline{\psi(g)}=\langle \psi,\chi\rangle,
 \end{equation*}
for every characters $\chi$, $\psi$.

\par The characters of the irreducible representations of $G$ are orthonormal with respect to this inner product. 

\begin{Theorem}\label{orthonormalrelationstheorem}
Let $G$ be a finite group and let $\kk$ be an algebraically closed field such that $\chara \kk$ does not divide $|G|$. Then the following facts hold.
 \begin{enumerate}
  \item A representation $V$ over $\kk$ is irreducible if and only if $\langle \chi_V,\chi_V\rangle=1$.
  \item If $\chi$ and $\psi$ are the characters of two non-isomorphic irreducible $\kk$-representations then $\langle \chi,\psi\rangle=0$.
 \end{enumerate}
\end{Theorem}

\begin{Cor}\label{orthonormalrelationscorollary}
 Let $\kk$ and $G$ be as in Theorem \ref{orthonormalrelationstheorem}. 
 Let $V$ be a $\kk$-representation of $G$ and let
 \begin{equation*}
  V=V_1^{n_1}\oplus\cdots\oplus V_r^{n_r}
 \end{equation*}
be its decomposition into irreducible representations $V_i$. 
If $\chi_V$ is the character of $V$ and $\chi_{V_i}$ is the character of $V_i$ then
\begin{equation*}
 n_i=\langle \chi_V,\chi_{V_i}\rangle.
\end{equation*}
In particular, two representations are isomorphic if and only if they have the same character.
 \end{Cor}

\par Thanks to Corollary \ref{orthonormalrelationscorollary}, characters are an important tool to understand the decomposition of a representation into irreducible representations.
We will use them to obtain the decomposition of the $q$-th symmetric powers $\Sym^q_{\kk}(V)$ of a two-dimensional faithful representation $V$ of a Klein group.
We begin with an elementary statement.

\begin{Lemma}\label{sumofrootofunitybounded}
 Let $\lambda\neq1$ be a root of unity in $\mathbb{C}$, and let $f:\mathbb{N}\rightarrow\mathbb{C}$ be the function  
 \begin{equation*}
 f(N):= \sum_{q=0}^N\sum_{t=0}^q\lambda^{2t-q}.
 \end{equation*}
Then $f(N)=o(N^2)$, that is
\begin{equation*}
\lim_{N\rightarrow+\infty}\frac{|f(N)|}{N^2}=0.
	\end{equation*}
\end{Lemma}

\begin{proof}
We have
\begin{equation*}
\begin{split}
0\leq\frac{|f(N)|}{N^2}&\leq\frac{2\left|\sum_{q=0}^N\sum_{t=0}^q\lambda^{2t-q}\right|}{N^2}=\frac{\left|\left(\sum_{q=0}^{N+1}\lambda^q\right)\left(\sum_{q=0}^{N+1}\lambda^{-q}\right)-\sum_{t=0}^{N+1}\lambda^{2t-N-1}\right|}{N^2}\\
&\leq\frac{4/|1-\lambda|^2+N+2}{N^2}, 
\end{split}
\end{equation*}	
which goes to $0$ for $N\rightarrow+\infty$.
\end{proof}

\begin{Ex}
 Let $(V,\rho)$ be a two-dimensional representation of a finite group $G$ over an algebraically closed field $\kk$.
 Since $\kk$ is algebraically closed, for every element $g\in G$ the matrix $\rho(g)$ can be diagonalized 
 \begin{equation*}
  \rho(g)=\begin{pmatrix}
           \lambda & 0\\
           0 & \mu
          \end{pmatrix},
 \end{equation*}
with $\lambda,\mu\in \kk$. 
The representation $\Sym^q(V)$ evaluated at the element $g$ is given by the matrix
\begin{equation*}
 \begin{pmatrix}
  \lambda^q &  0            & \cdots  &    \cdots            & 0 \\
    0       & \lambda^{q-1}\mu & 0       &   \cdots             & 0 \\
   \vdots   &                &  \ddots       &                 & \vdots \\
            & \cdots        &  \cdots & \lambda\mu^{q-1}   & 0 \\
    0        &  \cdots      &  \cdots        &      0           & \mu^{q}
   \end{pmatrix}.
\end{equation*}
In other words, if $\lambda$ and $\mu$ are the eigenvalues of $V$ at the element $g$, then the eigenvalues of $\Sym^q(V)$ at $g$ are $\{\lambda^{t}\mu^{q-t}: \ t=0,\dots,q\}$.
\end{Ex}

\begin{Lemma}\label{characterofSymqV}
 Let $\kk$ be an algebraically closed field and let $G$ be finite group such that $\chara\kk$ and $|G|$ are coprime. 
 Let $(V,\rho)$ be a two-dimensional representation of $G$ whose image is contained in $\Sl{2,\kk}$, then the (Brauer) character of the representation $\Sym^q(V)$ is given by 
 \begin{equation*}
 \chi_{\Sym^q(V)}(g)=\sum_{t=0}^q\lambda_g^{2t-q},
 \end{equation*}
where  $\lambda_g$ is the lift to $\mathbb{C}$ of an eigenvalue of the matrix $\rho(g)$.
\end{Lemma}

\begin{proof}
 Let $g\in G$ and let $\lambda_g$ and $\mu_g$ be the lift to $\mathbb{C}$ of the eigenvalues of $\rho(g)$. Since $\rho(g)\in\Sl{2,\kk}$, we have $\mu_g=\lambda_g^{-1}$. 
 Then the lift of the eigenvalues of $g$ in the representation $\Sym^q(V)$ are 
 \begin{equation*}
  \{\lambda_g^t\cdot\mu_g^{q-t}: \ t=0,\dots, q\}=\{\lambda_g^{2t-q}: \ t=0,\dots,q\},
 \end{equation*}
so the formula for the character of $\Sym^q(V)$ follows immediately.
\end{proof}

\begin{Rem}
 The character of the symmetric representation $\Sym^q(V)$ can be computed also using \emph{Molien's formula} (cf. \cite[Theorem 2.2.1]{Stu08})
 \begin{equation}\label{moliensformula}
  \sum_{q=0}^{+\infty}\chi_{\Sym^q(V)}(g)t^q=\frac{1}{\det(I-t\rho(g))}.
 \end{equation}
Here $t$ is an indeterminate, and the determinant on the right hand side is of a matrix with entries conveniently lifted to the polynomial ring $\mathbb{C}[t]$. 
Expanding the rational function on the right, we obtain a formal power series which is equal to the formal power series on the left, and we can use this equality to compute the character $\chi_{\Sym^q(V)}$. 
Notice that Molien's formula holds also if $\dim_{\kk}V>2$. 
\end{Rem}

\begin{Theorem}\label{theoremkleinrepresentation}
 Let $\kk$ be an algebraically closed field and let $G$ be a Klein subgroup of $\Sl{2,\kk}$, such that $\chara\kk$ does not divide $|G|$. 
 Let $(V,\rho)$ be a two-dimensional faithful $\kk$-representation of $G$ with image contained in $\Sl{2,\kk}$ and  let $(V_i,\rho_i)$ be an irreducible $\kk$-representation of $G$. 
 We denote by $\alpha_{i,q}(V)$ the multiplicity of $V_i$ in $\Sym^q(V)$ and by $\beta_{q}(V)=\dim_{\kk}\Sym^q(V)$. 
 Then
 \begin{equation*}
  \lim_{N\rightarrow+\infty}\frac{\sum_{q=0}^N\alpha_{i,q}(V)}{\sum_{q=0}^N\beta_{q}(V)}=\frac{\dim_{\kk}V_i}{|G|}.
 \end{equation*}
\end{Theorem}

\begin{proof}
Since $\dim_{\kk}V=2$, the dimension of $\Sym^q(V)$ is $q+1$. So the denominator in the limit above is $\sum_{q=0}^N\beta_{q}(V)=\frac12(N+1)(N+2)$.
From Corollary \ref{orthonormalrelationscorollary} and Lemma \ref{characterofSymqV} we have
\begin{equation*}
\begin{split}
 |G|\alpha_{i,q}(V)&=\langle \chi_{\Sym^q(V)},\chi_{V_i}\rangle=\sum_{g\in G}\chi_{\Sym^q(V)}(g)\cdot\overline{\chi_{V_i}(g)}\\
&=\sum_{g\in G}\overline{\chi_{V_i}(g)}\left(\sum_{t=0}^q\lambda_g^{2t-q}\right),
 \end{split}
\end{equation*}
where $\lambda_g$ is the lift to $\mathbb{C}$ of an eigenvalue of $\rho(g)$. 
The order $m_g$ of the root of unity $\lambda_g$ coincides with the order of $\rho(g)$ in $\Sl{2,\kk}$. Since $V$ is faithful, this coincides also with the order of the group element $g$ in $G$. 
In particular, $\lambda_g=1$ if and only if $g$ is the identity $I$ of $G$.
Thus, we can write the previous sum as
\begin{equation*}
\overline{\chi_{V_i}(I)}\left(\sum_{t=0}^q1\right)+\sum_{\substack{g\in G\\ g\neq I}}\overline{\chi_{V_i}(g)}\left(\sum_{t=0}^q\lambda_g^{2t-q}\right)= \dim_{\kk}(V_i)(q+1)+\sum_{\substack{g\in G\\ g\neq I}}\overline{\chi_{V_i}(g)}\left(\sum_{t=0}^q\lambda_g^{2t-q}\right).
\end{equation*}
We sum for $q$ running from $0$ to a fixed natural number $N$ and we obtain 
\begin{equation*}
\begin{split}
|G|\sum_{q=0}^N\alpha_{i,q}(V)=\frac12\dim_{\kk}(V_i)(N+1)(N+2)+o(N^2),
\end{split}
\end{equation*}
thanks to Lemma \ref{sumofrootofunitybounded}.
Therefore, the numerator of the limit is  
\begin{equation*}
\sum_{q=0}^N\alpha_{i,q}(V)=\frac{(N+1)(N+2)\dim_{\kk}V_i}{2|G|}+o(N^2)
\end{equation*}
so the limit is equal to $\frac{\dim_{\kk}V_i}{|G|}$ as desired.
\end{proof}

\begin{Rem}
 The multiplicity $\alpha_{i,q}(V)$ can be computed also using Molien's formula \eqref{moliensformula}, as showed by Springer in \cite{Spr87}. One obtains
 \begin{equation*}
 \sum_{q=0}^{+\infty} \alpha_{i,q}(V) t^q=\frac{1}{|G|}\sum_{g\in G}\frac{\chi_{V_i}(g^{-1})}{\det(I-t\rho(g))}.
 \end{equation*}
\end{Rem}

\subsection{Symmetric signature of Kleinian singularities}

\begin{Theorem}\label{theoremssKleinian}
Let $\kk$ be an algebraically closed field and let $G$ be a Klein subgroup of $\Sl{2,\kk}$ such that $\chara\kk$ does not divide $|G|$. 
Let $R$ be the corresponding Kleinian singularity and let $M_i$ be an indecomposable MCM $R$-module. 
Then the generalized symmetric signature and the generalized differential symmetric signature of $R$ with respect to $M_i$  are
\begin{equation*}
 s_{\sigma}(R,M_i)=s_{d\sigma}(R,M_i)=\frac{\rank_RM_i}{|G|}.
\end{equation*}
\end{Theorem}

\begin{proof}
Let $E$ be the fundamental module of $R$.
By Theorem \ref{Theoremsecondsyzygyfield1}, and by Theorem \ref{Zariskiisfundamental} we have
\begin{equation*}
 E\cong\Syz^2_R(\kk)\cong\Omega_{R/\kk}^{**}.
\end{equation*}
Therefore the two symmetric signatures, differential and syzygy, coincide for Kleinian singularities.
 
\par Let $V_i$ be the irreducible $\kk$-representation of $G$ such that $M_i=\mathcal{A}(V_i)$, and let $V$ be the fundamental representation of $G$, that is the $\kk$-representation such that $\mathcal{A}(V)=E$. 
Let $\alpha_{i,q}(E)$ be the multiplicity of $M_i$ into $\Sym_R^q(E)^{**}$ and let $\beta_q(E)=\rank_R\Sym_R^q(E)^{**}$.
 From Theorem \ref{symcommuteswithauslandertheorem} we have  
\begin{equation*}
\mathcal{A}'\big(\Sym^q_R(E)^{**}\big)=\Sym^q(V).
\end{equation*}
It follows that $\beta_q(E)=\dim_{\kk}\Sym^q(V)$ and $\alpha_{i,q}(E)$ equals the multiplicity of the representation $V_i$ in $\Sym^q(V)$. 
Since $V$ is the fundamental representation of a Klein group, it is faithful and its image is contained in $\Sl{2,\kk}$, so we can apply Theorem \ref{theoremkleinrepresentation} to conclude the proof.
\end{proof}

\begin{Cor}\label{corollaryKleinian}
 Let $\kk$ be an algebraically closed field and let $G$ be a Klein subgroup of $\Sl{2,\kk}$ such that $\chara\kk$ does not divide $|G|$. 
 The symmetric signature and the differential symmetric signature of the corresponding Kleinian singularity $R$ are
 \begin{equation*}
  s_{\sigma}(R)=s_{d\sigma}(R)=\frac{1}{|G|}.
 \end{equation*}
\end{Cor}

\begin{Ex}\label{exampleAnlimitnotexists}
Consider the $A_{n-1}$-singularity $R=\kk\llbracket u,v\rrbracket^{C_n}$ over an algebraically closed field $\kk$, and let $\mathcal{S}^q=\left(\Sym_R^q(\Syz_{R}^2(\kk))\right)^{**}$ as in the definition of symmetric signature.
From the proofs of Theorem \ref{theoremkleinrepresentation} and Theorem \ref{theoremssKleinian}, we have that $\rank_R\mathcal{S}^q=q+1$, and
\begin{equation*}
 \freerank_R\mathcal{S}^q=\frac{q+1}{n}(1+(-1)^q)+O(1)=\begin{cases}
                                                \frac{2(q+1)}{n}+O(1)\ &\text{ for }q \ \text{even}\\
                                                O(1)\  &\text{ for }q \ \text{odd}.
                                               \end{cases}
\end{equation*}
This shows that the limit 
\begin{equation*}
 \lim_{q\rightarrow+\infty}\frac{\freerank_R\mathcal{S}^q}{\rank_R\mathcal{S}^q}
\end{equation*}
does not exist.
Specifically, in this setting the syzygy module splits as $\Syz_R^2(\kk)\cong M_1\oplus M_{n-1}$ (with notation as in Example \ref{examplecyclicgroup3}).
Therefore we have 
\begin{equation*}
\mathcal{S}^q\cong \bigoplus_{t=0}^q M_1^{\otimes t}\otimes M_{n-1}^{\otimes q-t}\cong\bigoplus_{t=0}^qM_1^{\otimes 2t-q}.
\end{equation*}
\end{Ex}

\section{Symmetric signature of cones over elliptic curves}\label{sectionelliptic}
\par In this section we prove that the differential symmetric signature of the coordinate ring $R$ of a plane elliptic curve over an algebraically closed field of characteristic different from $2$ and $3$ is zero.
Our intention is motivated by the following remark, which shows that the F-signature of such rings (over a field of positive characteristic) is zero.

\begin{Rem}
 We consider $R=\kk[x,y,z]/(f)$, with $\kk$ an algebraically closed field of positive characteristic, and $f$ a homogeneous non-singular polynomial of degree $3$. 
 In other words, $R$ is the coordinate ring of a plane elliptic curve over $\kk$. 
 The ring $R$ is not strongly F-regular, and in particular its F-signature  $s(R)$ is $0$ by the result of Aberbach and Leuschke (Theorem \ref{AbeLeusTheorem}).
To see this, assume for simplicity that $f$ is in Weierstrass normal form, that is $f=y^2z-x^3-axz^2-bz^3$ for some $a,b\in\kk$. 
This is always possible if the characteristic of $\kk$ is different from $2$ and $3$. 
Then the ideal $I=(x,z)$ is not tightly closed in $R$, since $y^2\not\in I$, but $y^2\in I^{*}$. 
Therefore $R$ is not strongly F-regular and $s(R)=0$. 
\end{Rem}

\par The methods we use in this situation are different from those of previous sections, and are of geometric nature. 
We will use the correspondence between graded MCM $R$-modules and vector bundles over the smooth projective curve $Y=\mathrm{Proj}\,R$ to translate the problem into geometric language.
We will not make a general assumption on the characteristic of the base field $\kk$ in this section, but we will restrict to characteristic $\neq 2,3$ only when needed.
The letter $q$ is used to denote a positive integer.

\subsection{Generalities on vector bundles}
\par We recall some preliminary facts concerning vector bundles over curves, and in particular over elliptic curves.
For further details and the proofs of these facts, we refer to the books of Hartshorne \cite{Har77}, Le Potier \cite{LeP97}, and Mukai \cite[Chapter 10]{Muk03}.

\par Let $Y$ be a smooth projective curve over an algebraically closed field $\kk$.
We denote by $\mathrm{VB}(Y)$ the category whose objects are vector bundles over $Y$, and whose maps are morphisms between them. 
The category $\mathrm{VB}(Y)$ is a Krull-Schmidt category (cf. \cite[Theorem 3]{Ati56}).
Moreover there is an equivalence of categories between the category of vector bundles over $Y$ and the category of locally free sheaves on $Y$ which respects the rank (cf. \cite[Ex. II 5.18]{Har77}).
Having this identification in mind, we will use  indifferently the words vector bundle and locally-free sheaf. 
For example, we will use the notation for the structure sheaf $\mathcal{O}_Y$  of $Y$ also to denote the trivial bundle of rank one corresponding to it.

\par Let $\mathcal{S}$ be a locally free sheaf of rank $r$ over $X$. 
The degree of $\mathcal{S}$ is defined as the degree of the corresponding determinant line bundle $\mathrm{deg}\mathcal{S}=\deg\bigwedge^r\mathcal{S}$. 
The slope of $\mathcal{S}$ is $\mu(\mathcal{S})=\deg\mathcal{S}/r$. 
The degree is additive on short exact sequences and moreover $\mu(\mathcal{S}\otimes\mathcal{T})=\mu(\mathcal{S})+\mu(\mathcal{T})$.
\par The sheaf $\mathcal{S}$ is called \emph{semistable} if for every locally free subsheaf $\mathcal{T}\subseteq\mathcal{S}$ the inequality $\mu(\mathcal{T})\leq\mu(\mathcal{S})$ holds. 
If the strict inequality $\mu(\mathcal{T})< \mu(\mathcal{S})$ holds for every proper subsheaf $\mathcal{T}\subset\mathcal{S}$, then $\mathcal{S}$ is called \emph{stable}.
Stable vector bundles are indecomposable, but the converse does not hold in general.

\par The usual operations on vector spaces like direct sum, tensor product, symmetric and wedge powers extend naturally to vector bundles and correspond to the same constructions in the category of locally free sheaves.
The \emph{dual} of a vector bundle $E$ is the bundle $E^{\vee}:=\Hom(E,\mathcal{O}_Y)$. 

\begin{Rem} 
Since every vector bundle $E$ is canonically isomorphic to its double dual $E^{\vee\vee}$, reflexive symmetric powers $\Sym^q(-)^{\vee\vee}$ coincide with ordinary symmetric powers $\Sym^q(-)$ in the category $\mathrm{VB}(Y)$. 
\end{Rem}

The following are standard properties of symmetric powers of vector bundles (see also \cite[Corollary  6.4.14]{Laz04}).

\begin{Prop}\label{propertiessymvb}
 Let $Y$ be a smooth projective curve, let $E$ be a vector bundle and $\mathcal{L}$ a line bundle over $Y$. 
 Then the following facts hold.
\begin{compactenum}[1)] 
\item   $\Sym^q(E\otimes\mathcal{L})\cong \Sym^q(E)\otimes\mathcal{L}^{\otimes q}$.
 \item $\mu(\Sym^q(E))=q\cdot\mu(E)$.
\item If the field $\kk$ has characteristic $0$ and $E$ is semistable, then $\Sym^q(E)$ is also semistable.
 \end{compactenum}
\end{Prop}

\par Let $R$ be a normal standard-graded domain of dimension $2$ over an algebraically closed field $\kk$, that is $R_0=\kk$ and $R$ is generated by finitely many elements of degree $1$.  
The normal assumption on $R$  implies that $R_{\mathfrak{p}}$ is a regular ring for every prime ideal $\mathfrak{p}\neq R_{+}$. 
It follows that the projective variety $Y=\mathrm{Proj}\,R$ is a smooth projective curve over $\kk$. 
For every graded module $M$ we denote by $\widetilde{M}$ the corresponding coherent sheaf on $Y$. 
The sheaves $\widetilde{R(n)}$ are invertible (cf. \cite[Proposition II.5.12]{Har77}) and denoted by $\mathcal{O}_Y(n)$.   
Moreover, every MCM graded $R$-module $M$ is locally free on the punctured spectrum, so the associated coherent sheaf $\widetilde{M}$ is in fact a vector bundle over $Y$. 
If we denote  by $\cmodgr{R}$ the category of finitely generated graded MCM $R$-modules, then we have a functor
\begin{equation*}
\widetilde{ }:\cmodgr{R}\rightarrow\mathrm{VB}(Y),
\end{equation*}
whose properties we collect in the following proposition.

\begin{Prop}\label{propertiessheafification}
Let $R$ be a normal standard graded domain of dimension $2$ over a field $\kk$, let $Y=\mathrm{Proj}\,R$, and let $M$ and $N$ be finitely generated graded MCM $R$-modules. 
Then the following facts hold.
\begin{compactenum}[1)]
\item $\widetilde{M\oplus N}\cong \widetilde{M}\oplus\widetilde{N}$;
\item $\reallywidetilde{(M\otimes_R N)^{**}}\cong \widetilde{M}\otimes_{\mathcal{O}_Y}\widetilde{N}$;
\item $\reallywidetilde{\Sym_R^q(M)^{**}}\cong\Sym_Y^q(\widetilde{M})$.
\end{compactenum} 
\end{Prop}

\par Now we give an appropriate definition of free rank for the category of vector bundles over $Y$. 
 
\begin{Def}\label{freerankVB}
 Let $E$ be a vector bundle over $Y$ with a fixed very ample invertible sheaf $\mathcal{O}_Y(1)$. 
 We define the \emph{free rank} of $E$ as
  \begin{equation*}
  \begin{split}
 \freerank_{\mathcal{O}_Y(1)}(E):=\max\big\{n: \ \exists &\text{ a split surjection }\varphi: E\twoheadrightarrow F, \ \text{ with } F=\bigoplus_{i=1}^n\mathcal{O}_Y(-d_i) \\
 &\text{ a splitting vector bundle of rank } n\big\}. 
 \end{split}
 \end{equation*}
\end{Def}

\par Since the category $\mathrm{VB}(Y)$ has the KRS property, to compute the free rank of a bundle $E$, one should count how many copies of twisted structure sheaves $\mathcal{O}_Y(-d_i)$ appear in the decomposition of $E$ into indecomposable bundles.
For this reason, the free rank of vector bundles $\freerank_{\mathcal{O}_Y(1)}$ agrees with the graded free rank $\freerank_R^{\mathrm{gr}}$ and not with the free rank, as explained in the following corollary.

\begin{Cor}\label{rankfreerankcorollary}
 Let $R$ be a normal standard graded domain of dimension $2$ over a field $\kk$, let $Y=\mathrm{Proj}\,R$, and let $M$  be a finitely generated graded MCM $R$-module. 
 Then the following facts hold.
 \begin{compactenum}
  \item $\rank_RM=\rank_{\mathcal{O}_Y}\widetilde{M}$.
  \item $\freerank_R^{\textrm{gr}}M=\freerank_{\mathcal{O}_Y(1)}\widetilde{M}$.
 \end{compactenum}
\end{Cor}

\begin{Rem}\label{remarksyzygybundle}
Let $I$ be an $R_+$-primary homogeneous ideal with homogeneous generators $f_1,\dots,f_n$ of degrees $d_i$. 
The following presentation of $R/I$
\begin{equation*}
\bigoplus_{i=1}^nR(-d_i)\xrightarrow{f_1,\dots,f_n}R\rightarrow R/I\rightarrow0 
\end{equation*}
induces a short exact sequence of sheaves on $Y$
\begin{equation*}
0\rightarrow\mathcal{S}\rightarrow\bigoplus_{i=1}^n\mathcal{O}_Y(-d_i)\rightarrow\mathcal{O}_Y\rightarrow0,
\end{equation*}
where  $R/I$ corresponds to $0$ because $I$ is $R_+$-primary.
Since $\mathcal{O}_Y$ and $\bigoplus_{i=1}^n\mathcal{O}_Y(-d_i)$ are locally free,  the sheaf $\mathcal{S}$ is also locally free.
The sheaf $\mathcal{S}$ is the kernel of the sheaf morphism $f_1,\dots, f_n$ and  is denoted by $\Syz(f_1,\dots,f_n)$ and called \emph{syzygy bundle}. 
In fact, it is nothing but the vector bundle corresponding to the MCM graded $R$-module $\Syz_R^2(R/I)=\Syz_R^1(f_1,\dots,f_n)$.
\end{Rem}

\par Syzygies of the irrelevant maximal ideal $R_+$ play a special role, they correspond to the restriction of the cotangent bundle $\Omega_{\mathbb{P}^n}$ of $\mathbb{P}^n$ to the curve.
In fact, it follows from the Euler sequence on $\mathbb{P}^n$ \cite[Theorem 8.13]{Har77}
\begin{equation*}
  0\rightarrow \Omega_{\mathbb{P}^n}\rightarrow \bigoplus_{i=0}^n\mathcal{O}_{\mathbb{P}^n}(-1)\rightarrow\mathcal{O}_{\mathbb{P}^n}\rightarrow0,
 \end{equation*}
that $\Omega_{\mathbb{P}^n}$ is a syzygy bundle. So if $x_0,\dots,x_n$ is a system of generators for $\mathcal{O}_{\mathbb{P}^n}(1)$, we have the isomorphism
\begin{equation*}
  \Syz(x_0,\dots,x_n)\cong\Omega_{\mathbb{P}^n}.
 \end{equation*}
If we restrict these bundles to the curve $Y$, we obtain isomorphic bundles.
In fact, we have just proved the following proposition.
 
\begin{Prop}\label{equivsyzcotang}
 Let $Y$ be a smooth projective curve with an embedding in $\mathbb{P}^n$ given by $\varphi:Y\rightarrow\mathbb{P}^n$. 
 Let $x_0,\dots,x_n$ be a system of generators of $\mathcal{O}_{\mathbb{P}^n}(1)$ and denote by $\Omega_{\mathbb{P}^n}$ the cotangent bundle of $\mathbb{P}^n$ and by $\Omega_{\mathbb{P}^n}|_{Y}$ its restriction to $Y$.
 Then, we have the following isomorphism of vector bundles on $Y$ 
 \begin{equation*}
  \Syz(x_0,\dots,x_n)|_Y\cong\Omega_{\mathbb{P}^n}|_{Y}.
 \end{equation*}
\end{Prop}

\par It is a classical result, usually ascribed to Grothendieck, that the only indecomposable vector bundles over a curve of genus zero are of the form $\mathcal{O}_Y(a)$ for some integer $a$.
Therefore the structure of the category $\mathrm{VB}(Y)$ is quite easy for such a curve.
The next interesting case concerns a curve of genus $1$, an elliptic curve.
This case was treated by Atiyah \cite{Ati57}, who was able to give a description of all indecomposable vector bundles of fixed rank and degree.
Moreover he gave explicit formulas for the multiplicative structure of the monoid $\mathrm{VB}(Y)$ with the tensor product operation if the base field has characteristic zero.

\par In the following theorem we collect some results of Atiyah that we will need later on.

\begin{Theorem}[Atiyah, \cite{Ati57}]\label{Atiyahfundamentalbundle}
 Let $Y$ be an elliptic curve over an algebraically closed field $\kk$ with structure sheaf $\mathcal{O}_Y$.
 For every positive integer $r$ there exists a unique (up to isomorphism) indecomposable vector bundle  $F_r$ of rank $r$ and degree $0$ with $\Gamma(Y,F_r)\neq0$. 
 Moreover, the following facts hold.
 \begin{compactenum}
\item $F_r$ has only one section up to scalar multiplication, i.e. $\Gamma(Y,F_r)\cong\kk$.
\item $F_r\cong F_r^{\vee}$.
\item There is a non-split exact sequence
\begin{equation*}
 0\rightarrow \mathcal{O}_Y\rightarrow F_r\rightarrow F_{r-1}\rightarrow0.
\end{equation*}
\item $\Sym^{q}(F_2)\cong F_{q+1}$.
\item If $E$ is an indecomposable vector bundle of rank $r$ and degree $0$ then $E\cong L\otimes F_r$, where $L$ is a line bundle of degree zero, unique up to isomorphism and such that $L\cong\det E$. 
 \end{compactenum}
\end{Theorem}

\par We call the bundle $F_r$ \emph{Atiyah bundle of rank }$r$. 
Clearly we have that $F_1=\mathcal{O}_Y$, the structure sheaf is the unique line bundle of degree $0$ with non-zero sections. 
The bundle $F_2$ is given by the unique non-trivial extension

\begin{equation*}
0\rightarrow \mathcal{O}_Y\rightarrow F_2\rightarrow \mathcal{O}_Y \rightarrow 0.
\end{equation*}
In other words, $F_2$ is the non-zero element of $\mathrm{Ext}^1(\mathcal{O}_Y,\mathcal{O}_Y)$.

\par The following two results are well known to experts. 
However, since we are unable to locate a proof in the literature, for the sake of completeness we add a proof here.

\begin{Lemma}\label{lemmaSymqFrsemistable}
Let $Y$ be an elliptic curve over an algebraically closed field $\kk$ with structure sheaf $\mathcal{O}_Y$.
Then for every positive integers $r$ and $q$ the vector bundle $\Sym^q(F_r)$ is semistable.
\end{Lemma}

\begin{proof}
Indecomposable bundles over an elliptic curve are semistable (cf. \cite[Lemma 29]{Tu93}), so if $\chara\kk=0$ the result follows from Proposition \ref{propertiessymvb}.
\par Now assume that $\chara\kk=p>0$. 
The Atiyah bundle $F_r$ is defined inductively by a non-split short exact sequence  $0\rightarrow \mathcal{O}_Y\rightarrow F_r\rightarrow F_{r-1}\rightarrow0$, hence by a non-zero cohomology class $c\in H^1(Y,F_{r-1}^{\vee})\cong H^1(Y,F_{r-1})$.
By \cite[Lemma 2.4]{Bre05} the $p^e$-multiplication $[p^e]:Y\rightarrow Y$ on the elliptic curve kills this cohomology class for a suitable choice of $e$, that is $[p^e]^*(c)=0$.
Therefore the pullback via $[p^e]$ of the previous sequence splits, and inductively this shows that $[p^e]^*(F_r)$ is free.
Hence its symmetric powers are also free and in particular semistable.
Then the semistability of $[p^e]^*(\Sym^q(F_r))\cong\Sym^q([p^e]^*(F_r))$ implies that $\Sym^q(F_r)$ is semistable.
\end{proof}

\begin{Lemma}\label{lemmasemistableelliptic}
Let $Y$ be an elliptic curve over an algebraically closed field $\kk$, and let $E$ be a semistable vector bundle, then $\Sym^q(E)$ is semistable.
\end{Lemma}

\begin{proof}
If $\chara\kk=0$, then the result follows from Proposition \ref{propertiessymvb}.
So we assume that $\chara\kk=p>0$, and we denote by $r$ the rank of $E$.
By a result of Oda \cite{Oda71} (see also \cite{Tu93}) the bundle $E$ is strongly semistable, so by \cite[Proposition 5.1]{Miy87} it follows that $[n]^*E$ is semistable for every integer $n$, where $[n]:Y\rightarrow Y$ is the standard multiplication on the elliptic curve.
\par We consider the semistable bundle $[r]^*E$. 
Its degree is a multiple of $r$, so by tensoring with an appropriate line bundle $L$, we obtain a semistable vector bundle $E'=[r]^*E\otimes L$ of degree $0$.
It follows that the indecomposable direct summands of $E'$ must have degree $0$ too, so they are of the form $F_s\otimes \mathcal{L}$, with $\mathcal{L}$ line bundle.
With the help of a further multiplication map as in the proof of Lemma \ref{lemmaSymqFrsemistable}  we may achieve that $[n]^*(E)$ is a direct sum of line bundles of degree $0$. Hence its symmetric powers are semistable and this descends to the symmetric powers of $E$.
\end{proof}

\subsection{Symmetric signature of cones over elliptic curves}
\par Let $f$ be a homogeneous non-singular polynomial $f$ of degree $3$ in $\kk[x,y,z]$, and let $R=\kk[x,y,z]/(f)$. 
Then $R$ is a normal standard graded $\kk$-domain of dimension $2$ and the projective curve $Y=\mathrm{Proj}\,R$ is a plane elliptic curve over $\kk$.

\par To compute the symmetric signature of $R$ one should consider  reflexive symmetric powers of the second syzygy of the residue field, $\Syz^2_R(\kk)$. 
Thanks to Remark \ref{remarksyzygybundle} and Proposition \ref{propertiessheafification} this is equivalent to consider the syzygy bundle $\Syz(x,y,z)$ over $Y$ and taking ordinary symmetric powers in the category $\mathrm{VB}(Y)$.
We recall that by Corollary \ref{rankfreerankcorollary} we have $\rank_RM=\rank_{\mathcal{O}_Y}\widetilde{M}$ and $\freerank^{\mathrm{gr}}_RM=\freerank_{\mathcal{O}_Y(1)}\widetilde{M}$ for every MCM graded $R$-module.
Therefore, we have that $s_{\sigma}(R)$ exists if and only if the limit
\begin{equation*}
 \lim_{N\rightarrow+\infty}\frac{\sum_{q=0}^N\freerank_{\mathcal{O}_Y(1)}\Sym^q(\Syz(x,y,z))} {\sum_{q=0}^N\rank_{\mathcal{O}_Y}\Sym^q(\Syz(x,y,z))}
\end{equation*}
exists, and in this case they coincide.
\par The same reasoning applies to the differential symmetric signature.
In this case, we should consider the sheaf associated to the module of Zariski differentials $\Omega_{R/\kk}^{**}$ on $Y$ and then take ordinary symmetric powers in the category $\mathrm{VB}(Y)$.
In other words, we consider the vector bundles
\begin{equation*}
 \Sym^q_Y\left( \widetilde{\Omega_{R/\kk}}\right).
\end{equation*}
Notice that we can forget about the double dual inside $\Sym_Y^q(-)$, thanks to Proposition \ref{propertiessheafification}.
The differential symmetric signature of the coordinate ring $R$ exists if and only if the limit 
\begin{equation}\label{cotangentellipticlimit}
 \lim_{N\rightarrow+\infty}\frac{\sum_{q=0}^N\freerank_{\mathcal{O}_Y(1)}\Sym^q_Y\left( \widetilde{\Omega_{R/\kk}}\right)} {\sum_{q=0}^N\rank_{\mathcal{O}_Y}\Sym^q_Y\left( \widetilde{\Omega_{R/\kk}}\right)}
\end{equation}
exists, and in this case they coincide.

\begin{Prop}\label{propSyziscotangent} 
Let $Y$ be a plane elliptic curve over a field $\kk$ of characteristic $\neq2,3$ with coordinate ring $\kk[x,y,z]/(f)$, where $f$ is a homogeneous polynomial of degree $3$.
Then we have
\begin{equation*} 
\Syz\left(\frac{\partial f}{\partial x},\frac{\partial f}{\partial y},\frac{\partial f}{\partial z}\right)(3)\cong F_2.
\end{equation*}
\end{Prop}
\begin{proof} It is enough to show that $\Syz\left(\frac{\partial f}{\partial x},\frac{\partial f}{\partial y},\frac{\partial f}{\partial z}\right)(3)$ corresponds to a non-zero element of $\mathrm{Ext}^1(\mathcal{O}_Y,\mathcal{O}_Y)$ with global sections, then the uniqueness of $F_2$ will imply the desired isomorphism.
\par We consider the following injective map of sheaves, $\varphi:\mathcal{O}_Y\rightarrow\mathcal{O}_Y(1)^{\oplus3}$, given by $\varphi(1)=(x,y,z)$.
From the Euler formula
\begin{equation*}
\frac{\partial f}{\partial x}x+\frac{\partial f}{\partial y}y+\frac{\partial f}{\partial z}z=3f,
\end{equation*}
which vanishes on $Y$, we obtain that the image of $\varphi$ is contained in the syzygy bundle $\Syz\left(\frac{\partial f}{\partial x},\frac{\partial f}{\partial y},\frac{\partial f}{\partial z}\right)(3)$.
Actually, since this image does not vanish anywhere, it defines a subbundle and hence also a quotient bundle, which is by rank and degree reasons the structure sheaf.
In other words, we have the following short exact sequence of vector bundles
\begin{equation}\label{sequenceF2}
0\rightarrow\mathcal{O}_Y\xrightarrow{\varphi}\Syz\left(\frac{\partial f}{\partial x},\frac{\partial f}{\partial y},\frac{\partial f}{\partial z}\right)(3)\rightarrow\mathcal{O}_Y\rightarrow0,
\end{equation}
\par It remains to prove that the sequence \eqref{sequenceF2} is non-split, equivalently $(x,y,z)$ is the unique non-zero global section. 
This can be done easily by explicit computations, for example assuming that $f$ is in Weierstrass normal form.
\end{proof}

\begin{Cor}\label{corolOmega=F2}
Let $\chara \kk\neq2,3$, and let $\widetilde{\Omega_{R/\kk}}$ be the sheaf version of the cotangent module $\Omega_{R/\kk}$ of the cone $R$ of $Y$, then we have an isomorphism of vector bundles
 \begin{equation*}
  \widetilde{\Omega_{R/\kk}}(-1)\cong F_2.
 \end{equation*}
\end{Cor}
\begin{proof}
 We recall that the cotangent module $\Omega_{R/\kk}$ is the graded $R$-module
\begin{equation*}
 \Omega_{R/\kk}=<\mathrm{d}x,\mathrm{d}y,\mathrm{d}z>/\left(\frac{\partial f}{\partial x}\mathrm{d}x+\frac{\partial f}{\partial y }\mathrm{d}y+\frac{\partial f}{\partial z}\mathrm{d}z\right).
\end{equation*}
In other words, $\Omega_{R/\kk}$ can be defined by the following short exact sequence of graded $R$-modules
\begin{equation}\label{cotangentsequence}
 0\rightarrow R(-2)\xrightarrow{\psi} R^{\oplus3}\xrightarrow{\varphi}\Omega_{R/\kk}\rightarrow0,
\end{equation}
where $\psi(1)=\left(\frac{\partial f}{\partial x},\frac{\partial f}{\partial y},\frac{\partial f}{\partial z}\right)$, and $\varphi$ sends the canonical basis to $\mathrm{d}x,\mathrm{d}y,\mathrm{d}z$.
\par Sequence \eqref{cotangentsequence} induces a short exact sequence of locally free sheaves on $Y$
\begin{equation*}
 0\rightarrow\mathcal{O}_Y(-2)\rightarrow\mathcal{O}_Y^{\oplus3}\rightarrow\widetilde{\Omega_{R/\kk}}\rightarrow0. 
\end{equation*}
We dualize this sequence and we get
\begin{equation*}
0\rightarrow  \widetilde{\Omega_{R/\kk}}^{\vee}\rightarrow\mathcal{O}_Y^{\oplus3}\rightarrow\mathcal{O}_Y(2)\rightarrow0,
\end{equation*}
where the last map is given by $\psi$. 
We obtain that 
\begin{equation*}
\widetilde{\Omega_{R/\kk}}^{\vee}\cong\Syz\left(\frac{\partial f}{\partial x},\frac{\partial f}{\partial y},\frac{\partial f}{\partial z}\right)(2).
\end{equation*}
So by the previous Proposition \ref{propSyziscotangent} we have $\widetilde{\Omega_{R/\kk}}^{\vee}(1)\cong F_2$ on $Y$, but the Atiyah bundle $F_2$ is self-dual, so we have $\widetilde{\Omega_{R/\kk}}(-1)\cong F_2^{\vee}\cong F_2$.
\end{proof}

\par We can compute the differential symmetric signature of the cone $R$.

\begin{Theorem}\label{Theocotangentelliptic}
 Let $Y$ be a plane elliptic curve over an algebraically closed field $\kk$ of characteristic $\neq2,3$ with coordinate ring $R$.
 Then the differential symmetric signature of $R$ is $s_{d\sigma}(R)=0$.
\end{Theorem}

\begin{proof}
 By the previous observations we should compute the limit \eqref{cotangentellipticlimit}.
 From Corollary \ref{corolOmega=F2}, we have $\widetilde{\Omega_{R/\kk}}\cong F_2\otimes\mathcal{O}_Y(1)$. 
 Therefore from part \textit{4)} of Theorem \ref{Atiyahfundamentalbundle} and from Proposition \ref{propertiessymvb} we obtain
 \begin{equation*}
  \Sym^q_Y\left( \widetilde{\Omega_{R/\kk}}\right)\cong\Sym_Y^q(F_2)\otimes\mathcal{O}_Y(q)\cong F_{q+1}\otimes\mathcal{O}_Y(q)
 \end{equation*}
for all $q\geq1$. 
So the module $\Sym^q_Y\left( \widetilde{\Omega_{R/\kk}}\right)$ is indecomposable, and in particular it has free rank $0$ and the claim follows.
\end{proof}

\begin{Prop}\label{propSyzstablebundle}
The vector bundle $\Syz(x,y,z)$ is stable of rank $2$ and degree $-9$. 
Moreover $\det\Syz(x,y,z)=\mathcal{O}_Y(-3)$.
\end{Prop}

\begin{proof}
 The syzygy bundle fits into a short exact sequence
 \begin{equation}\label{syzygysequenceelliptic}
  0\rightarrow \Syz(x,y,z)\rightarrow \bigoplus_{i=1}^3\mathcal{O}_Y(-1)\xrightarrow{x,y,z}\mathcal{O}_Y\rightarrow0.
 \end{equation}
So from the additivity of rank and degree, we immediately get that $\Syz(x,y,z)$ has rank $2$ and degree
\begin{equation*}
 \deg\Syz(x,y,z)=\left(\deg\bigoplus_{i=1}^3\mathcal{O}_Y(-1)-\deg\mathcal{O}_Y\right)\deg Y= -3\deg Y=-9.
\end{equation*}
\par For the indecomposable property, we have that $\Syz(x,y,z)$ equals the restriction of the cotangent bundle $\Omega_{\mathbb{P}^2}$ of $\mathbb{P}^2$ to $Y$ by Proposition \ref{equivsyzcotang}. 
This bundle is stable, and in particular indecomposable, by the result of Brenner and Hein \cite[Theorem 1.3]{BH06}. 
Finally, taking determinants of sequence \eqref{syzygysequenceelliptic} yields $\det\Syz(x,y,z)=\mathcal{O}_Y(-3)$.
\end{proof}

\begin{Rem}\label{remrankSymSyz}
 Since $\Syz(x,y,z)$ has rank $2$, we have that $\rank_{\mathcal{O}_Y}\Sym^q(\Syz(x,y,z))=q+1$.
\end{Rem}

\par The difficult part is to determine the free rank of $\Sym^q(\Syz(x,y,z))$.
 
\begin{Cor}\label{corqodd0freerank}
Let $q$ be odd, then the bundle $\Sym^q(\Syz(x,y,z))$ contains no direct summand of rank one. 
In particular,  $\freerank_{\mathcal{O}_Y(1)}\Sym^q(\Syz(x,y,z))=0$. 
\end{Cor}
\begin{proof}
From Proposition \ref{propertiessymvb}, Proposition \ref{propSyzstablebundle} and Lemma \ref{lemmasemistableelliptic} we know that $\Sym^q(\Syz(x,y,z))$ is semistable of slope $-\frac{9}{2}q$. Let $\mathcal{L}$ be a direct summand of rank one. Since $\Sym^q(\Syz(x,y,z))$ is semistable, we have $\mu(\mathcal{L})=-\frac{9}{2}q$, but $\mu(\mathcal{L})=\deg\mathcal{L}$ must be an integer. Since $q$ is odd we get a contradiction.  
\end{proof}

\begin{Cor}\label{corsleq12elliptic}
 If the symmetric signature of $R$ exists, then $s_{\sigma}(R)\leq\frac{1}{2}$. 
\end{Cor}

\begin{proof} For every $q\in\mathbb{N}$, let $a_q=\freerank_{\mathcal{O}_Y(1)}\Sym^q(\Syz(x,y,z))$, and let $b_q=\rank_{\mathcal{O}_Y}\Sym^q(\Syz(x,y,z))$.
 From the previous Corollary \ref{corqodd0freerank} and Remark \ref{remrankSymSyz}, we have $a_q=0$ for $q$ odd, and $b_q=q+1$ for all $q$.
 \par For every even $q$  we have the following inequalities
 \begin{equation*}
  a_q+a_{q+1}=a_{q}\leq b_q=\frac{1}{2}(b_q+b_{q})\leq\frac{1}{2}(b_q+b_{q+1}).
 \end{equation*}
It follows that 
\begin{equation*}
 \frac{\sum_{q=0}^Na_q}{\sum_{q=0}^Nb_q}\leq\frac{\frac{1}{2}\sum_{q=0}^Nb_q}{\sum_{q=0}^Nb_q}+\phi(N)=\frac{1}{2}+\phi(N),
\end{equation*}
where $\phi(N)=\frac{b_N}{\sum_{q=0}^Nb_q}$ if $N$ is even, and $0$ otherwise.
In both cases $\phi(N)\rightarrow 0$ if $N\rightarrow+\infty$, so the claim is proved.
\end{proof}

\begin{Que}
 Is it true that also $s_{\sigma}(R)=0$?
\end{Que}

\section*{Acknowledgements}
We would like to thank Winfried Bruns, Ragnar-Olaf Buchweitz, Helena Fischbacher-Weitz, Lukas Katth\"an, and Yusuke Nakajima for their interest and many valuable comments. We thank the referee for showing us how to simplify the proof of Theorem \ref{theoremkleinrepresentation} and pointing out Remark \ref{remarksecondsyzygyfield2}.

\end{document}